\documentclass[12pt]{amsart}
\usepackage{etex}

\usepackage{amsmath}
\usepackage{mathrsfs, euscript}
\usepackage{bbm}
\usepackage{amssymb}
\usepackage{amscd}
\usepackage[all,cmtip]{xy}
\usepackage{enumerate}
\usepackage{mathabx}
\usepackage{tikz}

\newtheorem{theorem}{Theorem}[section]

\newtheorem{lemma}[theorem]{Lemma}

\newtheorem{proposition}[theorem]{Proposition}

\newtheorem{example}[theorem]{Example}

\newtheorem{remark}[theorem]{Remark}
\newtheorem{definition}[theorem]{Definition}

\newtheorem{notation}[theorem]{Notation}

\newcommand{\A}{\mathcal{A}}

\newcommand{\C}{\mathbb{C}}

\newcommand{\range}{\mathcal{B}}
\newcommand{\domain}{\mathcal{A}}

\newcommand{\Domain}{\mathbb{A}}
\newcommand{\E}{\mathcal{E}}
\newcommand{\EE}{\mathbb{E}}

\newcommand{\N}{\mathbb{N}}
\newcommand{\M}{\mathcal{M}}
\newcommand{\MM}{{\bf M}}
\newcommand{\D}{\mathcal{D}}

\newcommand{\s}{\mathcal{S}}

\newcommand{\ii}{{\bf i}}

\begin{document}
\unitlength=1mm
\special{em:linewidth 0.4pt}
\linethickness{0.4pt}
\title[]{Operator valued  random matrices and asymptotic freeness}
\author{Weihua Liu}
\maketitle
\begin{abstract}

We show that the limit laws of random matrices, whose entries are conditionally independent operator valued random variables having  equal second moments proportional to the size of the matrices, are operator valued semicircular laws.
Furthermore, we prove an operator valued analogue of Voiculescu's asymptotic freeness theorem.
By replacing  conditional independence with Boolean independence, we show that the  limit laws of  the random matrices are Operator-valued Bernoulli laws.
\end{abstract}

\section{Introduction}

Random matrices were first used to study population by Wishart \cite{Wis}, where the moments of random matrices are computed.   
The natural question regarding their distribution was raised in the pioneering work of Wigner  \cite{Wig}. 
Wigner  showed  that the spectral distributions of the $N\times N$ selfadjoint matrices, whose entries are  complex valued random variables having mean $0$ and variance $1/N$ that are independent up to the symmetry constraint $h_{ij}=\bar h_{ij}$, converge to the semicircular law
$$\frac{1}{2\pi}\sqrt{4-t^2}dt.$$
On the other side,  to attack the isomorphism problem of free group von Neumann algebras,  Voiculescu introduced his free probability theory in which the central notion \lq\lq free independence\rq\rq   is seen as a noncommutative analogue of the classical probabilistic concept of \lq\lq independence\rq\rq for random variables
\cite{V7}.  
It is shown that the central limit law for freely independent random variables  is exactly the semicircular law. 
Based on the occurrence  of the semicircular law in both random matrices theory and free probability theory, Voiculescu  found an amazing connection between these two theories by the notion of \lq\lq asymptotic freeness\rq\rq \cite{V6}.  
It says that classical independent random matrices become asymptotically free in the large $N$ limit.  
With the help of this connection,  people could use tools from one theory to another. 
For instance,  free probability methods are used for random matrices' computation \cite{BMS,RE}. 
On the other hand, by studying the structures of random matrices, isomorphisms between  certain von Neumann algebras are found \cite{Dy1,Dy2,Dy3, Ru}.  See  \cite{AGZ}  \cite{MS} for more developments and applications of asymptotic freeness. Besides those direct applications of asymptotic freeness,  a more powerful notion of \lq\lq free entropy\rq\rq was derived from it \cite{V5,V4}.  It was shown that the spectral measures of Gaussian Random matrices satisfy a large deviation principle and Voiculescu' first version of free entropy plays the role of  the relative entropy in Sanov's Theorem\cite{BG}.

In the 1990's, Voiculescu generalized free probability to amalgamated free product of $C^*$-algebras over a subalgebra in \cite{V3}, replacing scalar-valued expectations to conditional  expectations onto a certain subalgebra. The output of the conditional expectation is an operator, thus people the call the generalized theory  operator valued  probability theory or $\range$-valued(also $A$-valued \cite{Sh4}) probability theory.
As in the scalar case, operator valued free probability is also a natural analogue of classical probability theory in many aspects.  
For instance,  operator-valued independence is a universal independence relation in the sense of  Speicher,   and  has central limit laws and many other properties in analogue of classical probability.   
See  \cite{Sp1}  for a combinatorial development and  see   \cite{V8,V9} for an analytic development of   operator valued free probability theory.   
For developments and applications  of operator valued free central limit law, see  \cite{Sh8} for the connection between operator valued central limit law and band random matrices and see \cite{Sh4,PS,Sh5} for constructing isomorphisms between von Neumann algebras.

Since the ranges of operator-valued conditional expectations are not necessarily scalars, in $C^*$-operator-valued probability,  the second moment of an operator-valued random variable is in general a completely positive map.
It must be mentioned here that completely positive maps are not merely mathematical generalizations of states on $C^*$-algebras or von Neumann algebras.  
Due to the landmark works of  von Neumann \cite{von} and Gleason\cite{Gl},  a von Neumann algebra together with a normal state is a suitable frame work for a quantum system whose dimension $\geq 3$. 
However, the generalization of Gleason's theorem to an operator valued setting by Busch \cite{Bu} seems to be  more natural because it works for quantum systems of  any dimension. 
On the other hand, as the development of modern technologies, e.g.  in areas of  quantum information and quantum computer,  the completely positive maps are the proper mathematical concept to describe quantum channels\cite{NC}. 
Also, completely positive maps are used to characterize  quantum entanglements\cite{HHH}. 
Therefore, the interest of  studying operator valued probability does not only lay in a theoretical level but also in potential applications in the other areas.

The main purpose of this paper is to give an operator-valued generalization of  Voiculescu' asymptotic freeness, i.e.
we will introduce an operator-valued random matrix model and prove a theorem of operator-valued asymptotic
freeness.  
 Even though there many ways to develop the property,  we will focus on a probabilistic way.
 One will see that we are not only generalize the frame work to operator valued probability.  
 Once we restrict our attention back to  scalar noncommutative probability, 
 our work are still generaliation of previous works of asymptotice freeness \cite{ Dy2, Ry, Sh7}  at least in the following aspects:  the diagonal elements  are not necessarily uniformly bounded,  the entries can be chose to be mixed q-Gaussian random variables.  The following is the  main theorem. The asymptotic freeness of  operator valued extended models  will be considered in Section 6. 

\begin{theorem}\normalfont\label{main}
Let $\range$ be a $C^*$-algebra with norm $\|\cdot\|$ and $(\domain,\E:\domain\rightarrow \range)$ be a $\range-$valued probability space.
        Let  $S$ be an index set and $\{a(i,j;n,s)|s\in S, 1\leq i\leq j \leq m, n\in \N\}$ be a family of $\range$-random variables from $\domain$  such that 
\begin{itemize}
\item[Y1)]  $a(i,j;s,n)=a(i,j;s,n)^*,$  for all $1\leq i,j\leq n$, $s\in S$,
\item[Y2)]  $\E[a(i,j;s,n))]=0,$ for all $1\leq i,j\leq n$, $s\in S$,
\item[Y3)]  $\E[a(i,j;s,n)\bullet a(j,i;s,n))]=\frac{1}{n}\eta_s(\bullet)$ is a completely positive map from $\range$ to $\range$, for all $1\leq i,j\leq n$, $s\in S$,
\item[Y4)] for each $m$, there exists an $M_m>0$ such that  $$\sup\limits_{\substack{s_1,...,s_m\in S\\ 1\leq i_1,\cdots i_m, j_1,\cdots, j_m
\leq n}}\|{a(i_1,j_1;s_1,n)b_1a(i_2,j_2;s_2,n)b_2\cdots b_{m-1}a(i_m,j_m;s_m,n)}\|\leq M_mn^{-m/2}\prod\limits_{k=1}^{m-1}\|b_k\|,$$
\item[Y5)] the family $\{a(i,j;n,s)|s\in S, 1\leq i\leq j \leq m\}$ of random variables are $\range-$valued conditionally independent.
\end{itemize}
Let  $$ Y(s,n)=\sum\limits_{1\leq i,j, \leq n}  a(i,j;s,n)\otimes e(i,j;n)\in \domain \otimes M_n.$$
In addition,  let 
$D(t,n)=\sum\limits_{1\leq i \leq n}  b(i;t,n)\otimes e(i,i,n)\in\range\otimes M_n$ for  $t$ taking values in some set $T$ such that 
\begin{itemize}
\item[D1)] the joint distribution of $\{D(t,n)\}_{t\in T}$ converges weakly in norm.
\item[D2)] $$\lim\limits_{n\rightarrow \infty}\frac{\|D(t,n)\|^k}{n}=0,$$ for all $t\in T$, $k\in\N$.
\item[D3)] $$ \limsup\frac{\sum\limits_{i=1}^n\|b(i;t,n)\|^k}{n}< \infty,$$ for all $t\in T$, $k\in\N$.
\end{itemize}

For each $n$, let $\E_n:\domain\otimes M_n(\C)$ be the map that 
$$\E_n[(a_{i,j})_{i,j=1,\cdots,n}]=\frac{1}{n}\sum\limits_{i=1}^n\E[a_{i,i}].$$
Then the joint distributions of the  family of sets of random variables 
$$\{Y(s,n)\}_{s\in S}\cup\{D(t,n)|t\in T\}$$
with respect to $\E_n$ converge weakly in norm to the joint distribution of  the family of  $$\{Y_s\}_{s\in S}\cup\{D_t|t\in T\}$$  such that the family of subsets $\{ \{Y_s\}|s\in S \}\cup\{\{D_t|t\in T\}\}$ are freely independent. 
Moreover,  for each $s\in S$,  the distribution of $Y_s$ is a $\range$-valued semicircular law with variance $\eta_s.$
\end{theorem}

Let us briefly explain the conditions of Theorem \ref{main}. 
Condition Y1) ensures that the random matrices $Y(s,n)$ are selfadjoint.  
Conditions Y2)and Y3) means that , for each $s\in S$.  $Y(s,n)$ has mean zero entries and equal second moments.   
Condition Y4) means that the family of random matrices have finite mixed moments in all degrees.  
In Condition Y5), the conditional independence realtion  is a natural noncommutative analogue of classical independence relation in the viewpoint of the distributional symmetries, because with one more assuption that random variables are identically distributed,  by an extended de Finetti theorem, it is equivalent to the exchangeability for the infinite sequences of random variables in scalar noncommutaive probability spaces with  faithful states\cite{Ko}.
Moreover, Condition Y5) is not only satisfied by freely independent  and classically independent random variables, but also for $q$-Gaussian even mixed $q-$Gaussian random variables\cite{BKS, BMS,BS1,BS2,BS3}.  To the diagonal elements $D(t,n)$, condition D1) is an operator valued analogue of the weak convergence .
Conditions D2) and D3) allow the diagonal elements to  slowly approach an unbounded operators.  To satisfy D2), one can choose $\|D(t,n)\|\leq f(\ln n)$ where $f$ is a polynomial. Therefore,  classical Gaussian random variables can be approached by $D(t,n)$.

In the realm of noncommutative probability,  exchangeable sequences of random variables are not limited to classical independent, freely independent and q-Gaussian random variables.
 If we allow an absence of the unit, then we have the so-called  Boolean independence which we can use to construct an exchangeable sequence of random variables \cite{Le,Liu,SW}.
  Therefore, we can construct symmetric random matrices with Boolean random matrices. 
   We will show that the large $N$ limit of Boolean random matrices does not converge to the semicircular law but to the Bernoulli law.

Besides the firs in the first introduction section, the rest of paper is organized as follows.
In Section 2,  we introduce necessary notations and definitions. 
In Section 3,  we develop some  combinatorial  results for proving the main result.
In Section 4,  we prove Theorem \ref{main}.
In Section 5,  we study the extend matrix model in the sense of Dykema and generalize Theorem \ref{main}. 
In Section 6,  we study  $N$ limit laws of Boolean random matrices.

\section{preliminaries}

In this section, we recall some necessary definitions and notations in free probability. We will start with operator valued settings.  For people who are interested in scalar case, one just need to replace the ranges of the conditional expectations by the set of complex numbers.  For more details, see \cite{NS, VDN} for the scalars setting  and see \cite{PV,Sp1} for operator-valued settings.

\begin{definition} \normalfont A $\range$-valued probability space $(\domain, \E:\domain\rightarrow \range)$ consists of a unital algebra  $\range$, a unital algebra $\domain$ which is also a $\range$-bimodule and a conditional expectation $\E:\domain\rightarrow \range$ i.e.
 $$\E[b_1ab_2]=b_1\E[a]b_2,  \,\,\, \E[b 1_\domain]=b,$$
for all $b_1,b_2,b\in\range$, $a\in\A$ and $1_\domain$ is the unit of $\domain$.  The elements of $\A$ are called $\range$-valued random variables.
Suppose that  $\range$ is a unital $C^*$-algebra and $\domain$ is a $*$-algebra,   the conditional expectation $\E$ is said to be  positive if 
 $$\E[aa^*]\geq 0,$$
 for  all $a\in \domain.$ An element $x\in \domain$ is selfadjoint if $x=x^*$.
\end{definition}

 Let $1_\range$ be the unit of $\range$, one should be careful that $1_\range1_\domain$ is not necessarily  equal to$1_\domain$. 
Because $\E[\range1_\domain] =\range$ is injective, $\range$ can be considered as the subalgebra $\range 1_\domain$ of  $\domain$.  
 In addition, if $1_\range1_\domain=1_\domain$, $\range$ can be considered as a unital subalgebra of $\domain$.

\begin{example}\normalfont Let $\domain=M_n(\range) $ be the algebra of $n\times n$ matrices over $\range$.   Let $\E_1:\domain\rightarrow \range$ be a map such that
$$\E_1[(b_{ij})^n_{i,j=1}]=\frac{1}{n}\sum\limits_{k=1}^n b_{ii},$$
for all $ (b_{ij})^n_{i,j=1}\in \domain$, and let  $E_2:\domain\rightarrow \range$ be a map such that
$$\E_2[(b_{ij})^n_{i,j=1}]= b_{11},$$
for all $ (b_{ij})^n_{i,j=1}\in \domain$. 
Then, one can easily check that both $(\domain, \E_1:\domain\rightarrow \range)$  and $(\domain, \E_2:\domain\rightarrow \range)$  are well defined $\range$-valued probability space. 
For $(\domain, \E_2:\domain\rightarrow \range)$,  there are different ways to embed $\range$ into $\domain$.  For example, $\range$ can be considered as $\range\otimes e_{11}$, where $e_{11}$ is the $n\times n $ matrix whose $(1,1)$ entry is $1$ and all the other entries are $0$, or $\range\otimes I_n$, where $I_n$ is the unit of $M_n(\range)$.
\end{example}

Given a $\range$-bimodule $*$-algebra $\domain$, then $M_n(\domain)=\domain\otimes M_n(\C)$ is again a $\range$-bimodule $*$-algebra with the bimodule actions
 $$b_1(a_{i,j})_{i,j=1}^nb_2=(b_1a_{i,j}b_2)_{i,j=1}^n,$$
for all $b_1,b_2\in\range$.
Given a $\range$-bimodule map $\E:\domain\rightarrow \range$, then it has a natural extension to $\E_n:M_n(\domain)\rightarrow \range$ that 
$$E_n[(a_{i,j})_{i,j=1}]=\frac{1}{n}\sum\limits_{k=1}^nE[a_{kk}],$$  
which is also a $\range$-bimodule map. 

Throughout the paper, $\range$ is always $C^*$-algebra, $\domain$ is $*$-algebra and  the conditional expectation $E$ is positive. 
  Recall that the map $\E$ is a said to be completely positive if $\E_n$ is positive for all $n\in\N$.  
  It follows Exercise 3.18 from \cite{Pa} that $\E$ is completely positive if $\E$ is a positive conditional expectation.  Therefore, for $n\in \N$, $(M_n(\domain),\E_n:M_n(\domain)\rightarrow \range)$ is a well-defined $\range$-valued probability space.

We denote by $\range\langle X_i, X^*_i|i\in I\rangle$  the $*$-algebra freely generated by $\range$ and non-commuting indeterminants $\{X_i, X^*_i|i\in I\}$, where $I$ is an index set.  
$\range\langle X_i, X^*_i|i\in I\rangle$ has a natural $*$-operation which send $X_i$ to $X_i^*$.
The elements in $\range\langle X_i, X^*_i|i\in I\rangle$  are called $\range$-polynomials. 
In addition,  $\range\langle X_i, X^*_i|i\in I\rangle_0$ denotes the subalgebra of $\range\langle X_i, X^*_i|i\in I\rangle$ which does not contain a constant term.

\begin{definition}\normalfont Given a $\range$-valued probability space  $(\domain, \E:\domain\rightarrow \range)$ and a family of random variables $(x_i)_{i\in I}$ of  $\domain$.  
The joint distribution of $(x_i)_{i\in I}$  is $\range$-linear map $\mu: \range\langle X_i, X^*_i|i\in I\rangle\rightarrow \range$  such that for $m\in \N$,
$$\mu(b_1X_{i_1}^{\epsilon_1}b_1X_{i_2}^{\epsilon_2}\cdots b_mX_{i_m}^{\epsilon_m}b_{m+1})=\E[b_1x_{i_1}^{\epsilon_1}b_1x_{i_2}^{\epsilon_2}\cdots b_mx_{i_m}^{\epsilon_m}b_{m+1}],$$
where $b_1,\cdots, b_{m+1}\in \range$, $i_1,\cdots, i_m\in I$ and $\epsilon_1,\cdots,\epsilon_m\in \{1, *\}$.
\end{definition}

\begin{notation}\normalfont 
Given a family of random variables $(x_i)_{i\in I}$ from $\domain$ and a $\range$-valued polynomial $p\in\range\langle X_i, X^*_i|i\in I\rangle$, then $p((x_i)_{i\in I})$ is the element in $\domain$ obtained by
replacing indeterminants $X_i$ by $x_i$ for all $i\in I$. For example, given $p=b_1X_{i_1}b_1X_{i_2}\cdots b_mX_{i_m}b_{m+1}\in \range\langle X_i, X^*_i|i\in I\rangle$, then $p((x_i)_{i\in I})=b_1x_{i_1}b_1x_{i_2}\cdots b_mx_{i_m}b_{m+1}$.
\end{notation}

Now, we are ready  to introduce  independence relations.
\begin{definition}\normalfont
Let  $(\A,\E:\A\rightarrow \range)$ be a $\range$-valued probability space.
\begin{itemize}
\item
 A family of unital $*$-subalgebras $\{\A_i\supset \range \}_{i\in I}$  are said to be conditionally independent with respect to $\E$ if 
$$\E[a_1a_2a_3]=\E[a_1\E[a_2]a_3],$$
whenever  $a_1,a_3\in alg_\range\{A_i|i\in I_1\}$ and $a_2\in alg_\range\{A_i|i\in I_2\}$ such that $I_1\cap I_2=\emptyset$, where $alg_\range\{A_i|i\in I_k\}$ is the $*$-algebra generated by the family of subalgebras $\{A_i|i\in I_k\}$ and $\range$. A family of random variables $ (x_i)_{i\in I}$ are said to conditionally independent with respect to  $\E$, if the unital subalgebras $\{\A_i\}_{i\in I}$ which are generated by $x_i$ and $B$ respectively are conditionally  independent.

\item A family of unital $*$-subalgebras $\{\A_i\supset \range \}_{i\in I}$  are said to be freely independent with respect to $\E$ if 
$$\E[a_1\cdots a_n]=0,$$
whenever $i_1\neq i_2\neq \cdots\neq i_n$, $a_k\in \A_{i_k}$ and $\E[a_k]=0$ for all $k$. A family of $ (x_i)_{i\in I}$ are said to be freely independent over $\range$, if the unital subalgebras $\{\A_i\}_{i\in I}$ which are generated by $x_i$ and $B$ respectively are freely independent, or equivalently 
$$\E[P_1(x_{i_1})P_2(x_{i_2})\cdots P_n(x_{i_n})]=0,$$
whenever $i_1\neq i_2\neq \cdots\neq i_n$, $P_1,...,P_n\in \range\langle X,X^*\rangle$ and $\E[P_k(x_{i_k})]=0$ for all $k$.\\

\item A family of non-unital $*$-subalgebras $\{\A_i\supset \range \}_{i\in I}$  is said to be boolean independent with respect to $E$ if 
$$\E[a_1\cdots a_n]=\E[a_1]\E[a_2]\cdots \E[a_n],$$
whenever  $a_k\in \A_{i_k}$ and $i_1\neq i_2\neq\cdots\neq i_n$.    A family of random variables $\{x_i\}_{i\in I}$ are said to be boolean independent over $\range$, if the non-unital subalgebras $\{\A_i\}_{i\in I}$ which are generated by $x_i$ and $B$ respectively are boolean independent, or equivalently 
$$\E[P_1(x_{i_1})P_2(x_{i_2})\cdots P_n(x_{i_n})]=\E[P_1(x_{i_1})]\E[P_2(x_{i_2})]\cdots \E[P_n(x_{i_n})],$$
whenever $i_1\neq i_2\neq\cdots\neq i_n$ and $P_1,...,P_n\in\range\langle X,X^*\rangle_0$. 
\end{itemize}
\end{definition}

\begin{remark}\normalfont
Conditional independence relation is not a universal relation in the sense of Speicher \cite{Sp} for the reason that  some mixed moments of the conditionally independent random variables are not uniquely determined by their pure moments.  
Also, it is shown  that classical independence, free independence and Boolean independence are the only commutative universal independence relations in scalar noncommutative probability. In addition, free independence and classical independence are special case of conditional independence, but Boolean independence is not.
\end{remark}

\begin{definition}\normalfont
For each $n$,  let $\{x_{n,i}|i\in I\}$ be a family of random variables  from a $\range$-valued probability space $(\domain, \E_n:\domain_n\rightarrow \range)$ and let $\mu_n$ be their joint distribution. We say that the sequence $(\mu_n)_{n\in \N}$ converges weakly in norm if 
the sequence $(\mu_n(p))_{n\in \N}\subset \range$ converges in norm   for all
for all $\range$-valued polynomial $p\in\range\langle X_i,X_i^*|i\in I \rangle$.
\end{definition} 

The definition is slightly different from it in \cite{CS} because we do not require the sequence $(\{x_{n,i}|i\in I\})_{n\in N}$ converge to a specific family of random variables in the sense of weakly-norm convergence. 
However, we have the following result which say that we can always construct a family of random variables to which  the sequence $(\{x_{n,i}|i\in I\})_{n\in N}$ converges.

\begin{proposition}\label{Existence of limit distribution}\normalfont
 For each $n$,  let $\{x_{n,i}|i\in I\}$ be a family of random variables  from a $\range$-valued probability space $(\domain, \E_n:\domain_n\rightarrow \range)$ and let $\mu_n$ be their joint distribution.  If the sequence $(\mu_n)_{n\in N}$  converges weakly in norm  then there exists a  family  of random variables $\{x_{i}|i\in I\}$,  from a $\range$-valued probability space $(\domain, \E:\domain\rightarrow \range)$, whose joint distribution is $\mu$,  such that   
$\mu_n(p)$ converges to $\mu(p)$ in norm for all $\range$-valued polynomial $p\in\range\langle X_i,X_i^*|i\in I \rangle$.
If $\domain_n$ are $C^*$-algebras and $\sup\limits_{i\in I,n\in \N}\|x(i,n)\|\leq \infty$,  then $\domain$ can be chosen as a $C^*$-algebra.
\end{proposition}

\begin{proof}
With out the loss of generality, for each $n$, we assume that $\domain_n$ is generated by $\range$ and $\{x(n,i)|i\in I\}$. 
Then  $\bigoplus\limits_{n=1}^{\infty}\domain_n $ is a $\range$-bimodule $*$-algebra and $x_i=(x_{n,i})_{n\in\N}\in\bigoplus\limits_{n=1}^{\infty}\domain_n$ for all $i\in I$. 
Let $\domain$ be the subalgebra of  $\bigoplus\limits_{n=1}^{\infty}\domain_n $, which is generated by $\{x_i|i\in I \}$ and $\range$. 
Notice that that for $(y_n)_{n\in \N}\in \domain$, there exists a $\range$-polynomial $p\in\range\langle X_i,X_i^*|i\in I\rangle$ such that $y_n=p(\{x(n,i)|i\in I\})$ for all $n\in \N$.  Since $\mu_n$ converges weakly in norm,  we can define a map   $\E:\domain\rightarrow \range$ such that
$$\E[(y_n)_{n\in \N}]=\lim\limits_{n\rightarrow \infty} \E_n[y_n],$$
for all $(y_n)_{n\in \N}\in \domain$. 
One can easily check that $\E$ is a conditional expectation from $\domain$ to $\range$.  Moreover, $\E$ is positive since $\E_n$ are positive.

Suppose that $\domain_n$ are $C^*$-algebras and $\sup\limits_{i\in I,n\in \N}\|x(i,n)\|\leq \infty$, then $\sup\limits_{n\in \N}\|p(\{x(n,i)|i\in I\})\| \leq \infty$ for all $p\in\range\langle X_i,X_i^*|i\in I\rangle $.  Therefore, for each $y=(y_n)_{n\in \N}\in \domain$, $\|y\|_\domain=\liminf\limits_{n\rightarrow \infty} \|y_n\|$ is well defined. One can easily check that $\|yy^*\|_\domain=\|y\|_\domain^2$, $\|y_1+y_2\|_\domain\leq \|y_1\|_\domain+\|y_2\|_\domain$ for all $y,y_1,y_2\in \domain.$  It follows that $\domain$ is a pre-$C^*$-algebra with respect to $\|\|_\domain$, e.g. the completion $\overline{\domain}^{\|\|_\domain}$ is a $C^*$-algebra. 
\end{proof}

\begin{definition}\normalfont
For each $n$,  let $\{x_{n,i}|i\in I\}$ be a family of random variables  from a $\range$-valued probability space $(\domain, E_n:\domain_n\rightarrow \range)$ and let $\mu_n$ be their joint distribution.  We say that the sequence of families of random variables $(\{x_{n,i}|i\in I\})_{n\in\N}$ are asymptotically free with amalgamation over $\range$ if  $\mu_n$ converges weakly in norm to the joint distribution of  a family of   freely independent random variables $\{x_{i}|i\in I\}$.
\end{definition}

\subsection{Partitions and cumulants}

\begin{definition}\label{partition}\normalfont
Let $S$ be an ordered set:
\begin{itemize}
\item[1.] A partition $\pi$ of a set $\s$ is a collection of disjoint, nonempty sets $V_1,...,V_r$ whose union is $\s$.  $V_1,...,V_r$ are called  blocks of $\pi$, and we denote by $|\pi|=r$ the number of blocks of $\pi$.  
Given  $s,t\in \s$,  we denote by $t\sim_\pi s$ if $s,t$ are contained in the same block  of $\pi.$ 
A block $V$ of $\pi$ is an interval if there is no triple $(s_1,s_2,r)$ such that $s_1<r<s_2$, $s_1,s_2\in V$, $r\not\in V$.  
The collection of all partitions of $\s$ will be denoted by $P(\s)$.   
We denote by $[m]$ the set $\{1,\cdots,m\}$, and denote by $P(m)$ short for $P([m])$.
\item[2.] Given two partitions $\pi$, $\sigma$, we say $\pi\leq \sigma$ if each block of $\pi$ is contained in a block of $\sigma$.
\item[3.]A partition $\pi\in P(\s)$ is noncrossing if there is no quadruple $(s_1,s_2,r_1,r_2)$ such that $s_1<r_1<s_2<r_2$, $s_1,s_2\in V$, $r_1,r_2\in W$ and $V,W$ are two different blocks of $\pi$.   
The set of noncrossing partitions  of  $\s$ will be denoted by $NC(\s)$.   
We also write $NC(m)$ short for $NC([m])$.
\item[4.]A partition $\pi\in P(\s)$ is an  interval partition if   every block of $\pi$ is an interval. 
 It is apparent that all interval partitions are noncrossing partitions.  
 The set of interval  partitions  of  $\s$ will be denoted by $IN(\s)$.   
 We also write $IN(m)$ short for $IN([m])$.  
\item[5.] Let ${\bf i}=(i_1,...,i_k)$ be a sequence of indices from $I$. We denote by $\ker \ii$ the element of $P(k)$ whose blocks are the equivalence classes of the relation
$$s\sim t\Leftrightarrow i_s=i_t$$
\item [6.] If $\pi\in P(\s)$ is a partition such that  every  block of $\pi$ contains exactly $2$ elements, then we call $\pi$ a pair partition.  We denote by $P_2(\s)$ set of pair partitions of $\s$.    The set of noncrossing pair partitions $P_2(m)\cap NC(m)$ will be denoted by $NC_2(m)$, and the set of interval pair partitions $P_2(m)\cap IN(m)$ will be denoted by $IN_2(m)$.
\end{itemize}
\end{definition}

\begin{remark}\normalfont
The noncrossing partitions can be defined recursively, i.e.  a partition $\pi\in P(\s)$ is noncrossing if and only if $\pi=\{\s\}$ or it has an interval block $V\neq \s$ such that $ \pi\setminus\{V\}$ is a noncrossing partition on $\s\setminus V$. Furthermore,  the noncrossing pair partitions can be defined recursively, i.e.  a partition $\pi\in P(\s)$ is a noncrossing pair partition if and only if $\pi=\{\s\}$ when $R$ has exactly $2$ elements or it has an interval block $V$ having exactly $2$ elements and $ \pi\setminus\{V\}$ is a noncrossing pair partition on $\s\setminus V$.
\end{remark}

\begin{definition}\normalfont Let $(\domain, \E:\domain\rightarrow\range)$ be a $\range$-valued probability space:
\begin{itemize}
\item[1.] For each $ n\in\N$, let $\rho^{(n)}$ be an $n$-$\range$-linear map $\rho^{(n)}:\A^{ \otimes_{\range}n}\rightarrow \range$, i.e.
$$\rho^{(n)}(b_0a_1b_1,a_2b_2,...,a_nb_n)=b_0\rho^{(n)}(a_1,b_1a_2,...,b_{n-1}a_n)b_n,$$
for all $b_0,...,b_n\in \range$.   For $m\in \N$ and $\pi\in NC(m)$, we define $\rho^{(\pi)}: \A^{ \otimes_{\range}m}\rightarrow \range$ recursively as follows. If $\pi$  has only one block, namely $[m]$, then 
$$\rho^{\pi}(a_1,a_2,...,a_m)=\rho^{(m)}(a_1,a_2,...,a_m),$$
for any $a_1,\cdots,a_n\in \domain$. Otherwise, we define
$$\rho^{(\pi)}(a_1,...,a_n)=\rho^{(\pi\setminus V)}(a_1,...,a_l\rho^{(s)}(a_{l+1,...,a_{l+s}}),a_{l+s+1},...,a_n),$$
for any $a_1,\cdots,a_n\in \domain$, where $V=(l+1,l+2,...,l+s)$ is an interval block of $\pi$.

\item[2.] Define the $\range$-valued moments function $\E^{(n)}:\A^{ \otimes_{\range}n}\rightarrow \range$ by 
$$\E^{(n)}(a_1,\cdots, a_n)=\E[a_1 a_2 \cdots a_n].$$  
 The $\range$-valued free cumulants function $\kappa_\E^{(\pi)}:\A^{ \otimes_{\range}m}\rightarrow \range$ are defined recursively for $\pi\in NC(m)$, $m\geq 1$, by the following equation: for each $n\in \N$ and $a_1,\cdots,a_n$, we have
 $$ \kappa_\E^{(n)}(a_1,\cdots,a_n)= \E[a_1\cdots a_n]-\sum\limits_{\pi\in NC(n),\pi\neq \{[n]\}} \kappa_\E^{(\pi)}(a_1,\cdots,a_n).$$
 Similarly,   the $\range$-valued Boolean cumulants function $b_\E^{(\pi)}:\A^{ \otimes_{\range}m}\rightarrow \range$ are defined recursively for $\pi\in IN(m)$, $m\geq 1$, by the following equation: for each $n\in \N$ and $a_1,\cdots,a_n$, we have
 $$ b_\E^{(n)}(a_1,\cdots,a_n)= \E[a_1\cdots a_n]-\sum\limits_{\pi\in IN(n),\pi\neq \{[n]\}} b_\E^{(\pi)}(a_1,\cdots,a_n).$$
\end{itemize}
\end{definition}

Due to Speicher,  free independence  can be characterized in terms of vanishing of mixed free cumulants.  Similarly  Boolean independence  can be characterized in terms of vanishing of mixed Boolean cumulants. 

\begin{theorem}\normalfont
Let $(\domain, \E:\domain\rightarrow \range)$ be a $\range$-valued probability space and $\{x_i|i\in I\}$ be a family of random variables in $\domain$. Then the family $\{x_i|i\in I\}$ is freely independent with amalgamation over $\range$ if and only if
$$ \kappa_\E^{(\pi)}(x_{i_1}^{\epsilon_1}b_1, x_{i_2}^{\epsilon_2}b_2,\cdots,x_{i_m}^{\epsilon_m}b_m)=0,$$
whenever $i_1,\cdots,i_m\in I$, $b_1,\cdots, b_m\in \range$, $d_1,\cdots, d_m\in \{1,*\}$ and $\pi\in NC(n)$ such that $\pi\not\leq \ker\ii$, where $\ii=(i_1,\cdots,i_m).$
Similarly, the family $\{x_i|i\in I\}$ is Boolean independent with amalgamation over $\range$ if and only if
$$ b_\E^{(\pi)}(x_{i_1}^{\epsilon_1}b_1, x_{i_2}^{\epsilon_2}b_2,\cdots,x_{i_m}^{\epsilon_m}b_m)=0,$$
whenever $i_1,\cdots,i_m\in I$, $b_1,\cdots, b_m\in \range$, $d_1,\cdots, d_m\in \{1,*\}$ and $\pi\in IN(n)$ such that $\pi\not\leq \ker\ii$, where $\ii=(i_1,\cdots,i_m).$
\end{theorem}

Now, we introduce  the combinatorial descriptions of $\range$-valued semicircular and $\range$-valued Bernoulli random variables, which are the free analogues  and the Boolean analogues of real Gaussian random variables, respectively.  

\begin{definition}\normalfont
Let $(\domain, \E:\domain\rightarrow \range)$ be a $\range$-valued probability space.
\begin{itemize}
\item  A family $\{Y(i)|i\in I\}$ of selfadjoint random variables in $\domain$ is said to for a $\range$-valued  free centered semicircular family if or any  $m\geq 1$,  $\ii=(i_1,\cdots,i_m)\in I^m$ and $b_1,b_2,\cdots,b_{m+1}$, we have
$$ \kappa_\E^{(\pi)}(b_1Y(i_1),b_2Y(i_2),\cdots,b_mY(i_m)b_{m+1})=0,$$
unless $\pi\in NC_2(k)$ and $\pi\leq \ker\ii$, where $\ii=(i_1,\cdots,i_m).$ In particular, the family $\{Y(i)|i\in I\}$ is freely independent with amalgamation over $\range$.

\item  A family $\{Y(i)|i\in I\}$ of selfadjoint random variables in $\domain$ is said to for a $\range$-valued  Boolean centered Bernoulli  family if or any  $m\geq 1$,  $\ii=(i_1,\cdots,i_m)\in I^m$ and $b_1,b_2,\cdots,b_{m+1}$, we have
$$ b_\E^{(\pi)}(b_1Y(i_1),b_2Y(i_2),\cdots,b_mY(i_m)b_{m+1})=0,$$
unless $\pi\in IN_2(k)$ and $\pi\leq \ker\ii$, where $\ii=(i_1,\cdots,i_m).$  In particular, the family $\{Y(i)|i\in I\}$ is Boolean independent with amalgamation over $\range$.
\end{itemize}
\end{definition}
\begin{example}\normalfont Notice that both $NC(2)$ and $IN(2)$ have two elements $\{\{1\}, \{2\}\} $ and $\{\{1,2\}\}$. Let $Y$ be a $\range$-valued random variables from $(\domain,\E:\domain\rightarrow \range)$.  Then, 
$$\kappa_\E^{(1)}(Y)=b_\E^{(1)}(Y)=\E[Y]$$
and
$$\kappa_\E^{(2)}(YbY)=b_\E^{(2)}(YbY)=\E[YbY]-\E[Y]b\E[Y],$$
for all $b\in \range.$  If $Y$ has mean zero, them $\kappa_\E^{(2)}(YbY)=b_\E^{(2)}(YbY)=\E[YbY]$.
If $Y$ is a $\range$-valued semicircular random variable, then
  $$\E[Yb1Yb_2Yb_3Y]= \E[Yb1Y]b_2\E[Yb_3Y]+ \E[Yb1\E[Yb_2Y]b_3Y],$$
  for all $b_1,\cdots,b_3\in\range.$
  
 If $Y$ is a $\range$-valued Bernoulli random variable, then
  $$\E[Yb1Yb_2Yb_3Y]= \E[Yb1Y]b_2\E[Yb_3Y],$$
  for all $b_1,\cdots,b_3\in\range.$
\end{example}
 See more examples of computations of the central limit laws in  \cite{BPV,Sp1}.

\section{Some combinatorial results}

In this section, we introduce several operations on partitions and prove some results that are required in proving Theorem \ref{main}.   See  \cite{Dy2,MS,NS,Ry} for more combinatorial results, in particular the estimations of  the numbers of blocks of some partitions, related to classical random matrix theory.

\begin{notation}\normalfont Let $m$ be a natural number.
\begin{itemize} 
\item Given $\sigma\in P(m+1)$ and $\pi \in P(m)$, we define $\pi\hookrightarrow \sigma\in P(2m+1)$ be  the partition obtained by partitioning the odd numbers $\{1,\cdots, 2m+1\}$ according to $\sigma$ and the even numbers $\{2,\cdots, 2m\}$ according to $\pi$.
\item Given $\pi_1, \pi_2 \in P(m)$, we define $\pi_1\wr\pi_2\in P(2m)$ be  the partition obtained by partitioning the odd numbers $\{1,\cdots, 2m-1\}$ according to $\pi_1$ and the even numbers $\{2,\cdots, 2m\}$ according to $\pi_2$.
\item  Given a partition $\pi\in P(\s)$ and $\s'$ is a subset of $\s$, then the restriction of $\pi$ to $\s'$ will denoted by $\pi(\s')$.
\end{itemize}
\end{notation}

\begin{example} Let $m=3$.  If  $\sigma=\{\{1,4\}, \{2,3\}\}\in P(4)$ and $\pi=\{\{1,2\},\{3\}\}$ then 
$\pi\hookrightarrow \sigma=\{ \{1, 7\}, \{3,5\}, \{2,4\}, \{6\}\}$, which can be seen from the following diagram. 

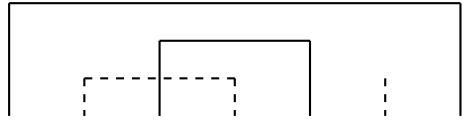
\begin{center}
  \thicklines

\begin{tikzpicture}[scale=1]

\draw [thick] (0,0)-- (6,0);
\draw [thick] (0,0)-- (0,-1.5);
\draw [thick] (6,0)-- (6,-1.5);
\draw [thick] (2,-0.5)-- (2,-1.5);
\draw [thick] (2,-0.5)-- (4,-0.5);
\draw [thick] (4,-0.5)-- (4,-1.5);
\draw [thick,dashed] (1,-1)-- (3,-1);
\draw [thick,dashed] (1,-1)-- (1,-1.5);
\draw [thick,dashed] (3,-1)-- (3,-1.5);
\draw [thick,dashed] (5,-1)-- (5,-1.5);

\put (-1, -18) {$1$};
\put (9, -18) {$\bar 1$};
\put (19, -18) {$2$};
\put (29, -18) {$\bar 2$};
\put (39, -18) {$3$};
\put (49, -18) {$\bar{3}$};
\put (59, -18) {$4$};
\end{tikzpicture}
\end{center} 
$\\\ $
\end{example}

\begin{definition}\normalfont
Let $\pi\in NC(m)$.  The  Kreweras complement $K(\pi)$ is the largest partition in $NC(m)$ such that $\pi\wr K(\pi)\in NC(2m)$,  the \emph{outer  Kreweras complement} $OK(\pi)$ is the largest partition in $NC(m+1)$ such that $\pi\hookrightarrow OK(\pi)\in NC(2m+1)$ and the \emph{inner Kreweras complement} $IK(\pi)$ is the largest partition in $NC(m-1)$ such that $IK(\pi)\hookrightarrow \pi\in NC(2m-1)$.
\end{definition}
\begin{example}\normalfont
Let $m=6$. If $\pi=\{\{1,4\},\{2,3\},\{5,6\}\}\in NC_2(6)$ then $OK(\pi)=\{\{1,5,7\},\{2,4\},\{3\},\{6\}\}$ which can be see from the following diagram.

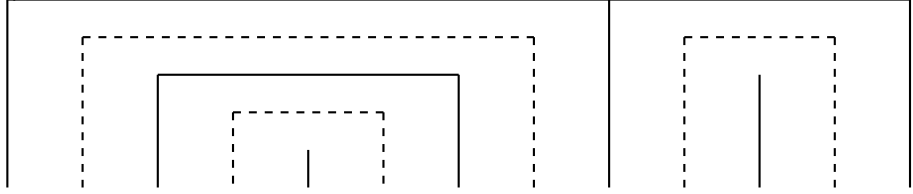
\begin{center}
  \thicklines

\begin{tikzpicture}[scale=1]

\draw [thick] (0,0)-- (12,0);
\draw [thick] (0,0)-- (0,-2.5);
\draw [thick] (12,0)-- (12,-2.5);
\draw [thick] (8,0)-- (8,-2.5);

\draw [thick] (4,-2)-- (4,-2.5);

\draw [thick,dashed] (5,-1.5)-- (5,-2.5);
\draw [thick,dashed] (3,-1.5)-- (3,-2.5);
\draw [thick,dashed] (3,-1.5)-- (5,-1.5);

\draw [thick] (6,-1)-- (6,-2.5);
\draw [thick] (2,-1)-- (2,-2.5);
\draw [thick] (2,-1)-- (6,-1);

\draw [thick,dashed] (7,-.5)-- (7,-2.5);
\draw [thick,dashed] (1,-.5)-- (1,-2.5);
\draw [thick,dashed] (1,-.5)-- (7,-.5);

\draw [thick,dashed] (11,-.5)-- (11,-2.5);
\draw [thick,dashed] (9,-.5)-- (9,-2.5);
\draw [thick,dashed] (9,-.5)-- (11,-.5);

\draw [thick] (10,-1)-- (10,-2.5);

\put (-1, -28) {$1$};
\put (9, -28) {$\bar 1$};
\put (19, -28) {$2$};
\put (29, -28) {$\bar 2$};
\put (39, -28) {$3$};
\put (49, -28) {$\bar{3}$};
\put (59, -28) {$4$};
\put (69, -28) {$\bar4$};
\put (79, -28) {$5$};
\put (89, -28) {$\bar 5$};
\put (99, -28) {$6$};
\put (109, -28) {$\bar 6$};
\put (119, -28) {$7$};
\end{tikzpicture}
\end{center} 
$\\ $
\end{example}

\begin{remark}\normalfont $K$ is a bijective \cite{NS},  hence  $IK$ is surjective and $OK$ is injective.
\end{remark}

\begin{definition}\normalfont
 A partition $\sigma\in P(m+1)$ is said to be closed if  $1$ and $(m+1)$ are in the same block of $\sigma$.
 We denote by $CP(m+1)$  the set of all closed partitions on $[m+1]$, and $CNC(m+1)$  the set of all closed noncrossing partitions on $[m+1]$.
\end{definition}
\begin{notation}\normalfont
   $\ii=(i_1,\cdots,i_m)\subseteq I$ or simply $\ii\ subset I$ means that $\ii$ is a sequence of indices from $I$. 
\end{notation}
If $\ii=(i_1,\cdots,i_m)\subseteq I$ is sequence of indices such that $\ker\ii\in CP(m)$, then $i_1=i_m$.

It is obvious that the map which send $\sigma\in CNC(m+1)$ to $\sigma([m])\in NC(m)$ is a bijection from $CNC(m+1)$ to $NC(m)$. Therefore, we have the following result.
\begin{lemma} 
The outer  Kreweras complement $OK$ is a bijection from $NC(m)$ to $CNC(m+1)$, whose inverse is $IK$.  
\end{lemma}

\begin{definition}Define 
$$d_m:P(m+1)\rightarrow P(m)$$ 
such that  for all $\sigma\in P(m+1)$,  $s\sim_{d_m(\sigma)} t $ if one of the following cases holds:
\begin{itemize}
\item $s\sim_{\sigma} t$ and $(s+1)\sim_{\sigma} (t+1)$
\item $s\sim_{\sigma} (t+1)$ and $(s+1)\sim_{\sigma} (t)$
\end{itemize}
\end{definition}

For convenience,  we denote by   $d: \amalg_{m\geq 1}P(m+1) \rightarrow \amalg_{m\geq 1}P(m)$ to be the map  such that $d|_{P(m+1)}=d_m$ for $m\geq 1$.

\begin{example}
Let $\sigma=\{\{1,5,7\},\{2,3,4\},\{6\}\}\in P(7)$. Then $d(\sigma)=\{\{1,4\},\{2,3\},\{5,6\}\}\in P(6)$.
\end{example}

\begin{lemma}\label{single entry}
 Given $\ii=(i_1,\cdots,i_m)\subseteq I$, let $\sigma=\ker \ii$ and $\pi=d(\sigma)$. Then, $s\sim_\pi t$ if and only if the two sets $\{i_s,i_{s+1}\}$ and $\{i_t,i_{t+1}\}$ are the same.
\end{lemma}
\begin{proof}
Suppose that $\{i_s,i_{s+1}\}$ and $\{i_t,i_{t+1}\}$ are the same if and only one of the following cases is true.
\begin{itemize}
\item $i_s=i_t$ and $i_{s+1}=i_{t+1}$,
\item  $i_s=i_{t+1}$ and $i_t=i_{t+1}$.
\end{itemize}
which is equivalent to that  one of the following cases holds:
\begin{itemize}
\item $s\sim_{\sigma} t$ and $(s+1)\sim_{\sigma} (t+1)$
\item $s\sim_{\sigma} (t+1)$ and $(s+1)\sim_{\sigma} (t)$.
\end{itemize}
Statements follows the definition of $d$.
\end{proof}

\begin{lemma}\label{remove 1}\normalfont
Let $\sigma\in P(m+1)$ such that $\{p\}$ is a block of $\sigma$ and $p-1\sim_\sigma p+1$ for some $p\in[m+1]$.  Then $|\sigma|=|\sigma([m+1]\setminus\{p,p+1\})|+1$ and  $\{p-1,p\}$ is a block of $d(\sigma)$.
\end{lemma}
\begin{proof}
Since $\{p\}$ is a  block of $\sigma$,  $|\sigma|=|\sigma([m+1]\setminus\{p\})|+1$.   
On the other hand,  $|\sigma([m+1]\setminus\{p\})|=|\sigma([m+1]\setminus\{p,p-1\})|$ because $p-1$ and $p+1$ are in the same block of $\sigma$. Therefore,  $|\sigma|=|\sigma([m+1]\setminus\{p,p+1\})|+1$.

Let $\pi=d(\sigma)$.   
According to the definition of $d$,  we have $p\sim_\pi (p-1)$. 
Suppose that  $p\sim_\pi q$ for some $q\neq p,p-1$. 
Then  we have  $p\sim_\sigma q$ or $p\sim_\sigma q+1$. 
Notice that none of $q$ and $q+1$ can be $p$,  it follows that $\{p\}$ is not a block of $\sigma$, which is a contradiction.  T
herefore, $\{p,p+1\}$ is a block of $d(\sigma)$.
\end{proof}

\begin{definition}\normalfont
We say that a partition $\sigma\in P(m+1)$ has property P	if $\sigma\in CP(m+1)$ and  the sizes of all blocks of $d(\sigma)$ are at least $2$.
\end{definition}

\begin{example}
 The partition $\{\{1,5,7\},\{2,3,4\},\{6\}\}\in P(7)$ has property P. The partition $\sigma'=\{\{1,4\},\{2,3\}\}\in P(4)$ does not have property P since $d(\sigma')=\{\{1,3\},\{2\}\}$ where $\{2\}$ is a block of size $1$.
\end{example}

Let $\s=\{n_1<n_2<\cdots <n_l\}$ be a subset of $\N$.  It is obvious that $P(\s)\cong P(l)$ and $NC(\s)\cong NC(l)$.
\begin{definition}\normalfont
 Let $\s=\{n_1<n_2<\cdots <n_l\}$.  A partition $\pi\in P(\s)$ has property P if $\pi$ has property P as an element in $P(l)$.
\end{definition}

\begin{lemma}\label{restriction of property P}\normalfont
Let $\sigma\in CP(m+1)$ be a closed partition such that $\{p\}$ is a block of $\sigma$ for some $p\in[m+1]$.
If $\sigma$ has property P, then $(p-1)\sim_\sigma (p+1)$ and $\sigma([m+1]\setminus\{p,p+1\})$ has property $P$.
\end{lemma}
\begin{proof}
Since $\sigma$ has property P,  the sizes blocks of $d(\sigma)$ are at least $2$.
If $(p-1)\sim_d(\sigma) q $ for some $q\neq p,p-1$, then $p\sim_\sigma q$ or $p\sim_\sigma q+1$. However, none of $q$ and $q+1$ can be $p$. It follows that $\{p\}$ is not a block of $\sigma$ which is a contradiction.
Therefore, $(p-1)\sim_d(\sigma) p$ and $p-1\sim_\sigma p+1$.

Because $\sigma$ is closed, we have $1\sim_\sigma (m+1)$.  
It follows that $1<p\leq m$. 
If $p=m$, then $1\sim_\sigma (p+1)\sim_\sigma (p-1)$. 
In this case, $[m+1]\setminus\{p,p+1\}=[m-1]$, hence $\sigma([m+1]\setminus\{p,p+1\})$ is closed.  
If $p\neq m$, then $m+1\in [m+1]\setminus\{p,p+1\}$. It follows that  $\sigma([m+1]\setminus\{p,p+1\})$ is  closed. 

Notice that $d(\sigma([m+1]\setminus\{p,p+1\}))=[d(\sigma)]([m]\setminus \{p-1,p\})$ the restriction of $d(\sigma) $ to $[m]\setminus \{p-1,p\}.$  By Lemma \ref{remove 1}, $\{p-1,p\}$ is a block of $d(\sigma)$.  Therefore,  the sizes of blocks of $[d(\sigma)]([m]\setminus \{p-1,p\})=d(\sigma)\setminus \{\{p-1,p\} \}$ are at least $2$. It follows that $\sigma([m+1]\setminus\{p,p+1\})$ has property $P$.
\end{proof}

\begin{lemma}\label{odd length}\normalfont Let $m\geq 3$ be an odd number, and let $\sigma\in CP(m+1)$ has property P. Then
$$|\sigma|\leq \frac{m+1}{2}.$$
\end{lemma}

\begin{proof}
When $m=3$,  $d(\sigma)\in P(3)$ whose blocks are of sizes at least $2$, thus $d(\sigma)$ has only one block. Since $1\sim_{d(\sigma)}2$, we have that 
$1\sim_\sigma 2$ or $1\sim_\sigma 3$. 
Therefore, $1,2,4$ are in the same block of $\sigma$ or $1,3,4$ are in the block of $\sigma$. 
Hence $\sigma$ has at most $2$ blocks. The statement is true in this case.

When $m=2k+1$ for $k\geq 2$, we have following  two cases:

Case 1.  $\sigma$ has no block of size $1$, then the statement follows the pigeonhole principle.

Case 2. Let $\{p\}$ be a block of $\sigma$.  Since $\sigma$ is closed, $p\neq 1, m+1$. By Lemma \ref{restriction of property P}, $\sigma([m+1]\setminus\{p,p+1\})$ has property P.  Since $\sigma([m+1]\setminus\{p,p+1\})$  can be viewed as an element in $P(m-1)$, we have 
$$|\sigma([m+1]\setminus\{p,p+1\})|\leq\frac{m-1}{2}.$$
By Lemma \ref{remove 1}, the proof is done.

\end{proof}

\begin{lemma}\label{must have single block}\normalfont Let $\pi\in NC_2(m)$ be a noncrossing pair partition. Then $\{p\}$ is a block of $OK(\pi)$ for some $p\in[m+1]$.
\end{lemma}
\begin{proof}
Since $\pi$ is a  a noncrossing pair partition, $\pi$ has an interval $\{p-1,p\}$ for some $p\in[m]$. According to the definition $OK$, $\{p\}$ is a block of $OK(\pi)$.
\end{proof}

\begin{lemma}\label{even length}\normalfont
 Let $m\geq 2$ be an odd number, and let $\sigma\in CP(m+1)$ has property P. Then
$$|\sigma|\leq \frac{m}{2}+1,$$
the equality holds if and only if  $\sigma=OK(\pi)$ for some $\pi\in NC_2(m)$.
\end{lemma}
\begin{proof}
When $m=2$. Since $\sigma\in CP(m+1)$, $\sigma$ can be either $\{ \{1,2,3\} \}$ or $\{\{1,3\}, \{2\}\}$. Both of them  have property P,  but just in the later case we have $|\sigma|=2=\frac{m}{2}+1$.  When $\sigma=\{\{1,3\}, \{2\}\}$, we have $d(\sigma)=\{\{1,2\}\}\in NC_2(2)$.  On the other hand, $\{\{1,2\}\}$ is the only element in $NC_2(2)$ and $OK(\sigma)=\{\{1,3\}, \{2\}\}$.

Suppose that the statement is true for $m=2k$. Now let us consider the situation for $m=2k+2$.  We have the following two cases:
Case 1.  $\sigma$ has no block of size $1$, then the statement  $|\sigma|\leq \frac{m+1}{2}<\frac{m}{2}+1$. On the other hand,  by Lemma \ref{must have single block}, $\sigma\neq OK(pi)$ for any $\pi\in NC_2(m)$.

Case 2.  $\sigma$ has a block of size $1$, say $\{p\}$. 
 Since $\sigma$ is closed, $p\neq 1, m+1$. 
 By Lemma \ref{restriction of property P}, $\sigma([m+1]\setminus\{p,p+1\})$ has property P.  
 Since $\sigma([m+1]\setminus\{p,p+1\})$  can be viewed as an element in $P(m-1)$, we have 
$$|\sigma|=|\sigma([m+1]\setminus\{p,p+1\})|+1\leq\frac{m-2}{2}+1+1=\frac{m}{2}+1.$$
The equality holds if and only if $\sigma([2k+1]\setminus\{p,p+1\})=OK(\pi')$ for some $\pi'\in NC_2(2k-2)$. 
By Lemma \ref{remove 1},  the equality holds if and only if $\sigma=OK(\pi)$ for some $\pi\in NC_2(2k)$.
\end{proof}

\begin{lemma} \label{inside OK} Let  $\pi\in NC_2(2k)$ be a noncrossing pair partition and $\sigma=OK(\pi)$.  Suppose that $\{p,q\}$ is a block of $ \pi$, then $OK(\pi([p+1,q-1]))=\sigma([p+1,q])$, where $\pi([p+1,q-1])$ is the restriction of $\pi$ to the interval $[p+1,q-1]$ and $\sigma([p+1,q])$ is the restriction of $\sigma$ to the interval $[p+1,q]$.
\end{lemma}
\begin{proof}
Because that  $p\sim_\pi q$, by the definition of $OK$,  all blocks $W$ of $\sigma$ are either completely contained in the interval $[p+1,q]$ or they are disjoint.  
Since $\sigma$ is the largest partition that
 $\pi\hookrightarrow \sigma$ is noncrossing, $\sigma([p+1,q])$ is the largest partition that
 $\pi([p+1,q-1])\hookrightarrow \sigma([p+1,q])$ is noncrossing. One can easily see from the following diagram. 

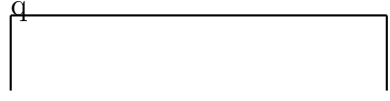
\begin{center}
\begin{tikzpicture}[scale=1]

\draw [thick] (0,0)-- (5,0);
\draw [thick] (0,0)-- (0,-1);
\draw [thick] (5,0)-- (5,-1);

\put (-7, -7) {$\cdots$};
\put (26, -14) {$\cdots$};
\put (53, -7) {$\cdots$};

\put (-1, -14) {$\bar p$};
\put (6, -14) {$(p+1)$};
\put (39, -14) {$ q$};
\put (49, -14) {$\bar q$};
\end{tikzpicture}
\end{center} 
$\ $
\end{proof}
From the above diagram,  we have the following result.

\begin{lemma} \label{anti-oriented paired}\normalfont Let  $\pi\in NC_2(2k)$ be a noncrossing pair partition and $\sigma=OK(\pi)$.  Given $p<q$,  if $p\sim\pi q$, then $p\sim_\sigma q+1$ and $p+1\sim_\sigma q$. 
\end{lemma}

 \section{proof of the main theorem}
 In this section, we are going to prove Theorem \ref{main}. 
 Before that,  we show some necessary facts that we can reduce the proof  into  a simpler cases.
 
Recall that, by  Proposition \ref{Existence of limit distribution},  there exists a family of random variables $\{D_t|t\in T\}$ from a $\range$-valued probability space, whose joint distribution is the weak-norm-limit of the distributions of $(\{D(t,n)|t\in T\})_{n\in N}$.  By Voiculescu's algebraic construction  \cite{V3} of free product with amalgamation over $\range$,  there exists a probability space $(\Domain,\EE:\Domain\rightarrow \range)$ which contains $\{D_t|t\in T\}$ and  a family of freely independent semicircular random variables $\{Y_s|s\in S\}$ such that $\EE(Y_sbY_s)=\eta_s(b)$ for all $b\in \range$, $s\in S$. Furthermore,  $\{D_t|t\in T\}$ and  $\{Y_s|s\in S\}$ are freely independent in $(\Domain,\EE:\Domain\rightarrow \range)$.

In the rest of this paper, $o(1)$ means a quantity that $\lim\limits_{n\rightarrow\infty}o(1)=0.$ 
Therefore,  to prove Theorem \ref{main}, we need to show that 
\begin{equation}\label{limit distribution}
\E_n\left[ P(\{D(t,n)|t\in T\},\{Y(s,n)|s\in S\})\right]-\EE\left[ P(\{D_t|t\in T\},\{Y_s|s\in S\})\right]=o(1).
\end{equation}
for all $\range$-valued polynomial $P\in \range\langle X_i, X_i^* |i\in S \coprod T\rangle$.  Let $\overline{\range}= \range\langle X_t, X_t^* |t\in  T\rangle$. Then we have 
$$  \range\langle X_i, X_i^* |i\in S \coprod T\rangle=\overline{\range} \langle X_s, X_s^* |s\in S\rangle.$$ 
Notice that the monomials of $\overline{\range} \langle X_s, X_s^* |s\in S\rangle$ are in the form 
$$ P_1X_{s_1}P_2X_{s_2}\cdots P_mX_{s_m}P_{m+1},$$
where $P_1,\cdots,P_{m+1}\in \range\langle X_t, X_t^* |t\in  T\rangle$ and $s_1,\cdots,s_m\in S$.
Therefore, to verify the Equation \ref{limit distribution},  we just need to check the following mixed moments.
 \begin{equation}\label{simplifying 1}
\begin{aligned}
&\E[P_1(\{D(t,n)|t\in T\})Y(s_1,n)\cdots Y(s_m,n)P_{m+1}(\{D(t,n)|t\in T\})]\\
=&\EE[P_1(\{D_t|t\in T\})Y_{s_1}\cdots Y_{s_m}P_{m+1}(\{D_t|t\in T\}] +o(1)
\end{aligned}
\end{equation}
for all $m\geq 1$, $P_1,\cdots, P_{m+1}\in\overline{\range}= \range\langle X_t, X_t^* |t\in  T\rangle$, $s_1,\cdots, s_m\in S$.

On the other hand, given a family of $\range$-polynomials $\{P_i|i\in I\}\subset \range\langle X_t,X_t^*|t\in T\rangle$, since $D(t,n)$ are diagonal elements of  $  M_n(\range)$,   $\{P_i( \{D(t,n)|t\in T\}) \}$  are also diagonal elements of  $  M_n(\range)$. Furthermore,  because  the sequence of families of random variables $ (\{D(t,n)|t\in T\})_{n\in \N}$ satisfy conditions D1), D2), D3) of Theorem \ref{main}, the sequence of families of random variables $(\{P_i( \{D(t,n)|t\in T\}) \})_{n\in \N}$  also satisfy the conditions D1), D2), D3) of Theorem \ref{main}. Namely,
\begin{itemize}
\item[D1)] The joint distribution of $\{P_i( \{D(t,n)|t\in T\})$ converges weakly in norm, because the compositions of $\range$-valued noncommutative polynomials are again $\range$-valued noncommutative polynomials.
\item[D2)] $$\frac{\|P_i( \{D(t,n)|t\in T\})\|^k}{n}=o(1)$$ for all $i\in I$, $k\in\N$, the $\range$-valued noncommutative polynomials are linear combinations of finite many $\range$-valued monomials .
\item[D3)] By using H\"older inequality, we have $$ \limsup\frac{\sum\limits_{j=1}^n\|P_i( \{b(j;t,n)|t\in T\})\|^k}{n}< \infty,$$ for all $i\in I$, $k\in\N$.
\end{itemize}

Therefore,  to verify the Equation \ref{simplifying 1}, we just need to show that 
\begin{equation}\label{simplifying 1}
\E[D(t_1,n)Y(s_1,n)\cdots Y(s_m,n)D(t_{m+1},n)] =\EE[D_{t_1}Y_{s_1}\cdots Y{s_m}D_{t_{m+1}}]+o(1)
\end{equation}
for all $t_1,\cdots,t_{m+1}\in T$ and $s_1,\cdots, s_m\in S$ under the conditions  of Theorem \ref{main}.

Now,  let us study the first term of the right side of Equation \ref{simplifying 1}.  According the moment-cumulant formula,  we have 
$$\EE[D_{t_1}Y_{s_1}\cdots D_{t_m}Y_{s_m}D_{t_{m+1}}]=\sum\limits_{\pi\in NC_2(m)}\sum\limits_{\substack{\sigma\in NC(m+1)\\ \pi\hookrightarrow \sigma\in NC(2m+1)}}\kappa_\EE^{(\pi\hookrightarrow \sigma)}(D_{t_1},Y_{s_1},\cdots ,D_{t_m},Y_{s_m},D_{t_{m+1}}).$$

\begin{definition}\normalfont  Given $\pi\in NC_2(m)$, we denote by $Out(\pi)$ be the set of all out blocks of $\pi$. Given $\pi_1,\pi_2\in NC_2(m)$,  we say that $\pi_1\stackrel{out}{\sim} \pi_2$ if $Out(\pi_1)=Out(\pi_2)$.
\end{definition}
\begin{notation}\normalfont  We denote by  $\bar{s}=(s_1,\cdots, s_m)\subset S^m$, $ \bar{t}=(t_1,\cdots, t_{m+1}) \subset T$ and $\ii=(i_1,\cdots,i_{m+1}) \subset [n]$.
\end{notation}

\begin{lemma}\label{recursive relation of limite distribution} \normalfont 
Let $\pi'\in NC_2(m)$ such that $\pi'\leq \ker \bar s$ and 
let $\{V_1=\{j_1,j_1'\},\cdots,V_l=\{j_l,j_l'\}\}$ be the family of outer blocks of $\pi'$. 
Suppose that $V_1,\cdots, V_l$ are ordered, i.e.  for $k=1,\cdots, l-1$, $j'_k+1=j_{k+1}$, then we have
\begin{align*}
&\EE[D_{t_1}E_{V_1}D_{t_{j'_1}} E_{V_2}\cdots E_{V_l}D_{t_{m+1}}]\\
=&\sum\limits_{\substack{\pi\in NC(m)\\ \pi\stackrel{out}{\sim} \pi'}}\sum\limits_{\substack{\sigma\in NC(m+1)\\ \pi\hookrightarrow \sigma\in NC(2m+1)}}\kappa_\EE^{(\pi\hookrightarrow \sigma)}(D_{t_1},Y_{s_1},\cdots ,D_{t_m},Y_{s_m},D_{t_{m+1}}),
\end{align*}
where  $E_{V_k}=\eta_{s_{j_k}}(\EE[D_{I_{j_k+1}}Y_{s_{j_k+1}}\cdots Y_{s_{j_k'-1}} D_{j_k'-1}])$ for $k=1,\cdots,l$.
\end{lemma}
\begin{proof}

Given a $\pi\in NC_2(m)$, then  $\pi\stackrel{out}{\sim} \pi'$ if and only if for all $k=1,\cdots,l $ the restriction of $\pi$ to the interval $[j_k+1,j_k'-1]$ is a noncrossing pair partition and $\{ j_k, j_k' \}$ is a block of $\pi$. Since, $V_1,\cdots, V_l$ are ordered and are all blocks of $\pi$, we have that  $j_1=1$ and $j_l'=m$.\\

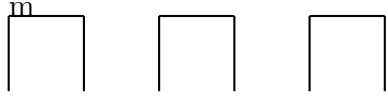
\begin{center}
\begin{tikzpicture}[scale=1]

\draw [thick] (0,0)-- (1,0);
\draw [thick] (2,0)-- (3,0);
\draw [thick] (4,0)-- (5,0);
\draw [thick] (0,0)-- (0,-1);
\draw [thick] (1,0)-- (1,-1);
\draw [thick] (2,0)-- (2,-1);
\draw [thick] (3,0)-- (3,-1);
\draw [thick] (4,0)-- (4,-1);
\draw [thick] (5,0)-- (5,-1);

\put (3, -7) {$\cdots$};

\put (23, -7) {$\cdots$};
\put (33, -7) {$\cdots$};
\put (43, -7) {$\cdots$};

\put (-6, -14) {$ 1$};
\put (14, -14) {$ j_2$};
\put (54, -14) {$ m+1$};
\put (-1, -14) {$\bar 1$};
\put (9, -14) {$\bar  j'_1$};
\put (19, -14) {$\bar j_2$};
\put (29, -14) {$\bar j'_2$};
\put (39, -14) {$\bar j_l$};
\put (49, -14) {$\bar m$};

\end{tikzpicture}
\end{center} 
$\\ $
From the above diagram, we see that if $\sigma\in NC(m+1)$ and $\pi\hookrightarrow \sigma\in NC(2m+1)$ then   the blocks of $\sigma$  are either a subset of $\{1,j_2,\cdots, j_l,m+1\}$ or a subset of $[j_k+1,j_k']$ for some $k$.  
Let us denote by $\bar \sigma$ the restriction  of $\sigma$ to $\{1,j_2,\cdots, j_l,m+1\}$. 
Then $\bar\sigma$ can  run over all $NC(1,j_2,\cdots, j_l,m+1)$, which is independent of the choice the restrictions of $\sigma$ to $[j_k+1,j_k']$ for all $k$. 
Let $\sigma_k$ be the restriction $\sigma$ to $[j_k+1,j_k']$. Then, we have the following equation for the sigma notation.
 $$\sum\limits_{\substack{\pi\in NC(m)\\ \pi\stackrel{out}{\sim} \pi'}}\sum\limits_{\substack{\sigma\in NC(m+1)\\ \pi\hookrightarrow \sigma\in NC(2m+1)}}
=\sum\limits_{\bar \sigma_1\in NC( 1,j_2,\cdots, j_l,m+1)}\sum\limits_{\substack{\sigma_1\in NC([j_1+1,j_1'])\\\pi_1\in NC_2([j_1+1,j'_1-1]) \\ \pi_1\hookrightarrow \sigma_1\in NC([2j_1+1,2j'_1-1])}}
\cdots \sum\limits_{\substack{\sigma_l\in NC([j_l+1,j_l'])\\\pi_l\in NC_2([j_l+1,j'_l-1]) \\ \pi_l\hookrightarrow \sigma_l\in NC([2j_l+1,2j'_l-1])}}.$$
It means that the following conditions are equivalent:
\begin{itemize}
\item[1.] $\pi$ is a noncrossing pair partition such that $ \pi\stackrel{out}{\sim} \pi'$ and $\sigma\in NC(m+1)$ is noncrossing partition such that $ \pi\hookrightarrow \sigma\in NC(2m+1)]$.
\item[2.] For each $k$, the restriction $\pi_k$ of $\pi$ to $[j_k+1,j_k']$ is a noncrossing partition and the restriction $\sigma_k$ of $\sigma$ to $[j_k+1,j_k']$ is a noncrossing partition such that $\pi_k\hookrightarrow \sigma_k$ is noncrossing.  In addition, $\bar\sigma$ is noncrossing.
\end{itemize}

Notice that $\kappa_\EE^{(2)}(Y_s,bY_s)=\EE[Y_sbY_s]=\eta_s(b)$ and $V_1,\cdots,V_l$ are blocks of $\pi$, by the recursive definition of $\kappa_\EE^{\pi}$, we have that
\begin{align*}
&\kappa_\EE^{(\pi\hookrightarrow \sigma)}(D_{t_1},Y_{s_1},\cdots ,D_{t_m},Y_{s_m},D_{t_{m+1}})\\
=&\kappa_\EE^{\sigma_1}(D_{t_1},\eta_{s_1}(\kappa_\EE^{(\pi_1\hookrightarrow \sigma_1)}(D_{2},\dots,D_{t_{j_1'}}))D_{j_2+1},\cdots, \eta_{s_{j_l}}(\kappa_\EE^{(\pi_l\hookrightarrow \sigma_l)}(D_{t_{j_l+1}},\dots,D_{t_{j_l'}}))D_{m+1}).
\end{align*}

For each $k$, we have 
$$ \sum\limits_{\substack{\sigma_k\in NC([j_l+1,j_l'])\\\pi_k\in NC_2([j_k+1,j'_k-1]) \\ \pi_l\hookrightarrow \sigma_l\in NC([2j_k+1,2j'_k-1])}}\kappa_\EE^{(\pi_k\hookrightarrow \sigma_k)}(D_{t_{j_k+1}},Y_{s_{j_k+1}}\cdots,D_{t_{j_k'}})=\EE[D_{t_{j_k+1}}Y_{s_{j_k+1}}\cdots D_{t_{j_k'}}].$$

It follows that 
\begin{align*}
&\sum\limits_{\substack{\pi\in NC(m)\\ \pi\stackrel{out}{\sim} \pi'}}\sum\limits_{\substack{\sigma\in NC(m+1)\\ \pi\hookrightarrow \sigma\in NC(2m+1)}}\kappa_\EE^{(\pi\hookrightarrow \sigma)}(D_{t_1},Y_{s_1},\cdots ,D_{t_m},Y_{s_m},D_{t_{m+1}})\\
=&\sum\limits_{\bar \sigma_1\in NC( 1,j_2,\cdots, j_l,m+1)}\kappa_\EE^{(\bar\sigma)}(D_{t_1},E_{V_1}D_{t_{j'_1}},\cdots ,E_{V_l}D_{t_{m+1}})\\
=&\EE[D_{t_1}E_{V_1}D_{t_{j'_1}} E_{V_2}\cdots E_{V_l}D_{t_{m+1}}].
\end{align*}
\end{proof}

Now, let us turn study the left hand side of Equation \ref{simplifying 1}, namely,
$$\E[D(t_1,n)Y(s_1,n)D(t_2,n)\cdots D(t_m,n)Y(s_m,n)D(t_{m+1},n)]. $$

The $(i,i)$th  entry is 
$$\sum\limits_{\substack{(i_1,\cdots,i_{m+1})\subset [n]\\i_1=i_{m+1}=i}} b(i_1;t_1,n)a(i_1,i_2;s_1,n) b(i_2; t_2,n)\cdots b(i_{m};t_{m},n) a(i_{m},i_{m+1};s_m,n)b(i_{m+1};t_{m+1},n).$$

\begin{notation}
 We denote by $Pro_n(\bar{s},\bar{t},\ii)$  the product
$$b(i_1;t_1,n)a(i_1,i_2;s_1,n) b(i_2; t_2,n)\cdots b(i_{m};t_{m},n) a(i_{m},i_{m+1};s_m,n)b(i_{m+1};t_{m+1},n). $$
\end{notation}

Therefore, we have 
\begin{equation}\label{simple equation 1}
\begin{aligned}
&\E_n[D(t_1,n)Y(s_1,n)D(t_2,n)Y(s_2,n)D(t_3,n)\cdots Y(s_m,n)D(t_{m+1},n)]\\
=&\frac{1}{n}\sum\limits_{\substack{\ii\subset[n]\\i_1=i_{m+1}}} \E[Pro_n(\bar{s},\bar{t},\ii) ]\\
=&\frac{1}{n}\sum\limits_{\sigma\in CP(m+1)}\sum\limits_{\substack{\ii\subset [n]\\ \ker\ii=\sigma} }\E[Pro_n(\bar{s},\bar{t},\ii)].
\end{aligned}
\end{equation}

If $\ker \ii=\sigma$ and $d(\sigma)$ has a block $\{p\}$ for some $p$,  by Lemma \ref{single entry}, then there is no $q\neq p$ such that $\{i_p,i_{p+1}\}=\{i_q,i_{q+1}\}$. It follows that $a(i_p,i_{p+1};s_p,n)$ is conditionally independent from all $a(i_q,i_{q+1};s_p,n)$ for $q\neq p$.  Since $\E[a(i_p,i_{p+1};s_p,n)]=0$, we would have $\E[Pro_n(\bar{s},\bar{t},\ii)]=0.$  Therefore, $\E[Pro_n(\bar{s},\bar{t},\ii)]$ is vanishing unless $\ker \ii$ has property P. By eliminating some vanishing mixed moments, we have

\begin{equation}\label{simple equation 2}
\begin{aligned}
&\E_n[D(t_1,n)Y(s_1,n)D(t_2,n)Y(s_2,n)D(t_3,n)\cdots Y(s_m,n)D(t_{m+1},n)]\\
=&\frac{1}{n}\sum\limits_{\substack{\sigma\in CP(m+1)\\ \sigma \text{has property P} }}\sum\limits_{\substack{\ii\subset [n]\\ \ker\ii=\sigma} }\E[Pro_n(\bar{s},\bar{t},\ii) ].\\
\end{aligned}
\end{equation}

\begin{notation}\normalfont Given a partition $\sigma=\{W_1,\cdots,W_l\}\in P(m+1)$,  $\ii(W_p)=k$ means  $ i_j=k$ if $j\in W_p$.

\end{notation}

\begin{lemma}\label{infinitesimal 1} If the number $|\sigma|$ of blocks of $\sigma$ is less than $\frac{m}{2}+1$, then
  $$\frac{1}{n}\sum\limits_{\substack{\ii\subset [n]\\ \ker\ii=\sigma} }\E[Pro_n(\bar{s},\bar{t},\ii)]=o(1).$$
\end{lemma}
\begin{proof}

By condition Y1) of the main theorem, we have 
$$ \E[Pro_n(\bar{s},\bar{t},\ii)]\leq M_m n^{-m/2}\prod\limits_{k=1}^{m+1}\|b(i_k;t_k,n) \|.$$
Let $\sigma=\{W_1,\cdots,W_l\}$.  Then we have 

\begin{align*}
&\|\frac{1}{n}\sum\limits_{\substack{\ii\in\subset [n]\\ \ker\ii=\sigma} }\E[Pro_n(\bar{s},\bar{t},\ii)]\|\\
\leq&\frac{1}{n}\sum\limits_{\substack{\ii\in\subset [n]\\ \ker\ii=\sigma} }M_m n^{-m/2}\prod\limits_{k=1}^{m+1}\|b(i_k;t_k,n)\| \\
\leq&M_m n^{-m/2-1}\sum\limits_{\substack{\ii\subset [n]\\ \ii(W_1),\cdots,\ii(W_l)\in[n]\\ i_{W_p}\neq i_{W_{q}}\, \text{if}\, p\neq q }}\prod\limits_{k=1}^{m+1}\|b(i_k;t_k,n)\| \\
\leq&M_m n^{-m/2-1}\sum\limits_{\substack{\ii\subset [n]\\ \ii(W_1),\cdots,\ii(W_l)\in[n]}}\prod\limits_{k=1}^{m+1}\|b(i_k;t_k,n)\| \\
=&M_m n^{-m/2-1}\prod\limits_{p=1}^l(\sum_{j=1}^n\prod\limits_{k\in W_p}\| b(j;t_k,n)\|)\\
=&n^{-m/2-1+l} M_m \prod\limits_{p=1}^l(\frac{1}{n}\sum_{j=1}^n\prod\limits_{k\in W_p}\| b(j;t_k,n)\|).\\
\end{align*}

By H\"older's inequality and the condition D3) of the main theorem, we have that 
$$\frac{1}{n}\sum_{j=1}^n\prod\limits_{k\in W_p}\| b(j;t_k,n)\|< \infty.$$

Since $l=|\sigma|< \frac{m}{2}+1$, the proof is done.
\end{proof}

Therefore,  we have 
\begin{align*}
&\E_n[D(t_1,n)Y(s_1,n)D(t_2,n)Y(s_2,n)D(t_3,n)\cdots Y(s_m,n)D(t_{m+1},n)]\\
=&\frac{1}{n}\sum\limits_{\substack{\sigma\in CP(m+1)\\ \sigma \text{has property P}\\ |\sigma|\geq m/2+1 }}\sum\limits_{\substack{\ii\subset [n]\\ \ker\ii=\sigma} }\E[Pro_n(\bar{s},\bar{t},\ii) ]+o(1).\\
\end{align*}

By Lemma \ref{even length}, the $\sigma\in CP(m+1)$ has property P and $|\sigma|\geq m/2+1$ if only if $\sigma=OK(\pi)$ for some noncrossing partition $\pi\in NC_2(m)$.  The previous equation becomes 

\begin{equation}
\begin{aligned}
&\E_n[D(t_1,n)Y(s_1,n)D(t_2,n)Y(s_2,n)D(t_3,n)\cdots Y(s_m,n)D(t_{m+1},n)]\\
=&\frac{1}{n}\sum\limits_{\pi\in NC_2(m)}\sum\limits_{\substack{\ii\subset [n]\\ \ker\ii=OK(\pi)} }\E[Pro_n(\bar{s},\bar{t},\ii) ]+o(1).\\
\end{aligned}
\end{equation}

\begin{lemma}\label{partial sum} Let $m$ be an even number. Then
$$\frac{1}{n}\sum\limits_{\pi\in NC_2(m)}\sum\limits_{\substack{\ii\subset [n]\\ \ker\ii=OK(\pi)} }\E[Pro_n(\bar{s},\bar{t},\ii) ]-\frac{1}{n}\sum\limits_{\pi\in NC_2(m)}\sum\limits_{\substack{\ii\subset [n]\setminus V\\ \ker\ii=OK(\pi)} }\E[Pro_n(\bar{s},\bar{t},\ii) ]=o(1),$$
for all   subset $V$ of $[n]$ such that $|V|<m$.
\end{lemma}
\begin{proof}
Fix $\pi\in NC_2(m)$, let $\sigma=OK(\pi)=\{W_1,\cdots,W_l\}$. Then $l=m/2+1$.  We have 

\begin{align*}
&\frac{1}{n}\|\sum\limits_{\substack{\ii\subset [n]\\ \ker\ii=\sigma} }\E[Pro_n(\bar{s},\bar{t},\ii) ]-\sum\limits_{\substack{\ii\in([n]\setminus V)^{m+1}\\\ \ker\ii=OK(\pi)} }\E[Pro_n(\bar{s},\bar{t},\ii)]\|\\
=&\frac{1}{n}\|\sum\limits_{\substack{ \ii(W_1),\cdots,\ii(W_l)\in[n]\\ i_{W_p}\neq i_{W_{q}}\, \text{if}\, p\neq q }}
\E[Pro_n(\bar{s},\bar{t},\ii) ]-\sum\limits_{\substack{ \ii(W_1),\cdots,\ii(W_l)\in[n]\setminus V\\ i_{W_p}\neq i_{W_{q}}\, \text{if}\, p\neq q }}\E[Pro_n(\bar{s},\bar{t},\ii) ]\|\\
=&\frac{1}{n}\|\sum\limits_{\substack{ \ii(W_1),\cdots,\ii(W_l)\in[n]\setminus V\\ i_{W_p}\neq i_{W_{q}}\, \text{if}\, p\neq q \\ W_p\in V\,\text{for some}\,p} }
\E[Pro_n(\bar{s},\bar{t},\ii)]\|\\
\leq&\frac{1}{n}\sum\limits_{\substack{ \ii(W_1),\cdots,\ii(W_l)\in[n]\setminus V\\ i_{W_p}\neq i_{W_{q}}\, \text{if}\, p\neq q \\ W_p\in V\,\text{for some}\,p} }
\|\E[Pro_n(\bar{s},\bar{t},\ii)]\|.
\end{align*}

Once we remove the restriction that  $i_{W_p}\neq i_{W_{q}}\, \text{if}\, p\neq q$ from the last sum, we will get more non-negative terms.  Therefore, we have 
\begin{align*}
\frac{1}{n}\sum\limits_{\substack{ \ii(W_1),\cdots,\ii(W_l)\in[n]\setminus V\\ i_{W_p}\neq i_{W_{q}}\, \text{if}\, p\neq q \\ W_p\in V\,\text{for some}\,p} }
\|\E[Pro_n(\bar{s},\bar{t},\ii)]\|\leq&\frac{1}{n}\sum\limits_{\substack{ \ii(W_1),\cdots,\ii(W_l)\in[n]\setminus V\\ W_p\in V\,\text{for some}\,p}}
\|\E[Pro_n(\bar{s},\bar{t},\ii)]\|.
\end{align*}
Notice that 
\begin{align*}
&\{( \ii(W_1),\cdots,\ii(W_l))\in[n]\setminus V| W_p\in V\,\text{for some}\,p\}\\
=&\bigcup\limits_{p=1}^l \{( \ii(W_1),\cdots,\ii(W_l))\in[n]| W_p\in V\}.
\end{align*}

Therefore, we have

\begin{align*}
&\frac{1}{n}\sum\limits_{\substack{ \ii(W_1),\cdots,\ii(W_l)\in[n]\setminus V\\ W_p\in V\,\text{for some}\,p}}
\|\E[Pro_n(\bar{s},\bar{t},\ii)]\|\\
\leq& \frac{1}{n}  \sum\limits_{p=1}^l \sum\limits_{\substack{ \ii(p)\in V\\ \ii(W_q)\in[n] \,\text{for}\, q\neq p }}( \|\E[Pro_n(\bar{s},\bar{t},\ii) ]\|)\\
\leq& \frac{1}{n}  \sum\limits_{p=1}^l \sum\limits_{\substack{ \ii(p)\in V\\ \ii(W_q)\in[n] \,\text{for}\, q\neq p }}M_m n^{-m/2}\prod\limits_{k=1}^{m+1}\|b(i_k;t_k,n)\|\\
=&  M_m\sum\limits_{p=1}^l \left\{\left(\frac{1}{n}\sum_{j\in V}\prod\limits_{k\in W_p}\| b(j;t_k,n)\|\right)\left(\prod\limits_{\substack{q=1,\cdots, l\\ q\neq p}}(\frac{1}{n}\sum_{j=1}^n\prod\limits_{k\in W_q}\| b(j;t_k,n)\|)\right)\right\}\\
\leq&  M_m\sum\limits_{p=1}^l \left\{\left(\frac{|V|}{n}\prod\limits_{k\in W_p}\| D(t_k,n)\|\right)\left(\prod\limits_{\substack{q=1,\cdots, l\\ q\neq p}}(\frac{1}{n}\sum_{j=1}^n\prod\limits_{k\in W_q}\| b(j;t_k,n)\|)\right)\right\}.
\end{align*}

By H\"older's inequality and the condition D2) of Theorem \ref{main}, we have  $$\frac{1}{n}\prod\limits_{k\in W_p}\| D(t_k,n)\|=o(1).$$ 
On the other hand, by H\"older's inequality and the condition D3) of Theorem \ref{main}, 
  $$\frac{1}{n}\sum_{j=1}^n\prod\limits_{k\in W_q}\| b(j;t_k,n)\|\leq \infty,$$
  for all $q.$ The proof is done.

\end{proof}

\begin{lemma}\label{recrusive moments }\normalfont
Let $m$ be an even number, and let $\pi\in NC_2(m)$. Suppose  $\ii\subset [n]$ such that $\ker \ii= OK(\pi)$. Then
\begin{align*}
\E[Pro_n(\bar{s},\bar{t},\ii)]=&\E^{(\pi)}[b(i_1;t_1,n)a(i_1,i_2;s_1,n), b(i_2; t_2,n)a(i_2,i_3;s_2,n),\cdots,\\
&b(i_{m};t_{m},n) a(i_{m},i_{m+1};s_m,n)b(i_{m+1};t_{m+1},n)].
\end{align*}
\end{lemma}

\begin{proof}
The statement is trivial when $m=2$
Suppose that it is true for $m=2k$. 
 For $m=2k+2$, since $\pi\in NC_2(2k+2)$, $\pi$ has an interval $\{p,p+1\}$ for some $p\in [2k+2]$.  
 Therefore,  $\{p\}$ and $\{p-1,p+1 \}$ are blocks of $\ker \ii$.  
By Lemma \ref{anti-oriented paired},  we have $i_{p-1}=i_{p+1}$.  In addition, we have that $i_p\neq i_q$ if $p\neq q$. 
It follows that 
$a(i_{p-1},i_p;s_{p-1},n)b(i_p; t_p,n)a(i_p,i_{p+1};s_p,n)$ is conditionally independent from  $a(i_q,i_{q+1};s_q,n)$ such that $q\neq p-1,p$. Therefore, we have
\begin{align*}
&\E[Pro_n(\bar{s},\bar{t},\ii)]\\
=&\E[b(i_1;t_1,n)a(i_1,i_2;s_1,n)\cdots b(i_{m};t_{m},n) a(i_{m},i_{m+1};s_m,n)b(i_{m+1};t_{m+1},n)]\\
=&\E[b(i_1;t_1,n)a(i_1,i_2;s_1,n)\cdots \E[a(i_{p-1},i_p;s_{p-1},n)b(i_p; t_p,n)a(i_p,i_{p+1};s_p,n)]\\  
&\cdots b(i_{m};t_{m},n) a(i_{m},i_{m+1};s_m,n)b(i_{m+1};t_{m+1},n)].
\end{align*}
By the inductive assumption, the proof is done.
\end{proof}

Since $a(i,j;s_1,n)$ and $a(i,j;s_2,n)$ are conditionally independent whenever $s_1\neq s_2$, $\E[Pro_n(\bar{s},\bar{t},\ii)]=0$ if $\ker \ii=OK(\pi)$ with $\pi\not\leq \bar s$. Therefore, we have 
\begin{equation}\label{simplifying 2}
\sum\limits_{\pi\in NC_2(m)}\sum\limits_{\substack{\ii\subset [n]\\ \ker\ii=OK(\pi)} }\E[Pro_n(\bar{s},\bar{t},\ii) ]=\sum\limits_{\substack{\pi\in NC_2(m)\\ \pi\leq \ker\bar s}}\sum\limits_{\substack{\ii\in\subset [n]\\ \ker\ii=OK(\pi)} }\E[Pro_n(\bar{s},\bar{t},\ii) ].
\end{equation}

 \begin{lemma} \label{mixed moments 3}
$$\sum\limits_{j=1}^n \E[a(i,j;s,n)D(j;t,n)a(i,j;s,n)]=\eta_s(\EE(D_{t}))+o(1),$$
for all $1\leq i\leq n$ and $s\in S$.
\end{lemma}
\begin{proof}
By the condition Y3) of Theorem \ref{main}, we have 
$$\sum\limits_{j=1}^n \E[a(i,j;s,n)D(j;t,n)a(i,j;s,n)]=\eta_s\left(\sum\limits_{j=1}^n\frac{1}{n}\E[D(j;t,n)]\right). $$
Since $$\eta_s\left(\sum\limits_{j=1}^n\frac{1}{n}\E[D(j;t,n)]\right)=\eta_s\left(\E_n[D(t,n)]\right), $$
and $ \E_n[D(t,n)] $ converges to $\EE(D_{t})$ in norm.  The statement follows.
\end{proof}

Now, we are ready to prove Theorem \ref{main}.
\begin{proof}
By Equation \ref{simple equation 1}, Equation \ref{simplifying 2}, Lemma \ref{infinitesimal 1}, it suffices to prove the following equation. 
$$\frac{1}{n}\sum\limits_{\substack{\pi\in NC_2(m)\\ \pi\leq \ker\bar s}}\sum\limits_{\substack{\ii\in\subset [n]\\ \ker\ii=OK(\pi)} }\E[Pro_n(\bar{s},\bar{t},\ii)]=\EE[D_{t_1}Y_{s_1}\cdots D_{t_m}Y_{s_m}D_{t_{m+1}}]+o(1).$$

When $m$ is an odd number, $NC_2(m)$ is empty. The equation is true.

When $m=2$, if $s_1\neq s_2$, then  we also get $0=0+o(1)$.

Suppose that $s_1=s_2=s$.   Since $Y_s$ is a $\range$-valued semicircular random variable, we have 
$$ \EE[D_{t_1}Y_{s}D_{t_2}Y_sD_{t_3}]=\EE[D_{t_1}\eta_s(\EE[D_{t_2}])D_{t_3}].$$

Then, by condition D1) of Theorem \ref{main}, we have 
\begin{align*}
\EE[D_{t_1}\eta_s(\EE[D_{t_2}])D_{t_3}]=\E_n[D(t_1,n)\eta_s(\EE[D_{t_2}])D(t_3,n)]+o(1).
\end{align*}
By Lemma \ref{mixed moments 3}, we have 

\begin{align*}
&\E_n[D(t_1,n)\eta_s(\EE[D_{t_2}])D(t_3,n)]\\
&=\frac{1}{n}\sum\limits_{i=1}^n\E[D(i;t_1,n)\left(\sum\limits_{j=1}^n\E[a(i,j;s,n)D(j;t,n)a(i,j;s,n)]+o(1)\right) D(i;t_3,n)]\\
&=\frac{1}{n}\sum\limits_{j=1}^n\sum\limits_{i=1}^n\E[D(i;t_1,n)\left(\E[a(i,j;s,n)D(j;t,n)a(i,j;s,n)]\right) D(i;t_3,n)]+o(1)\\
&=\frac{1}{n}\sum\limits_{j=1}^n\sum\limits_{i=1}^n\E[D(i;t_1,n)a(i,j;s,n)D(j;t,n)a(i,j;s,n) D(i;t_3,n)]+o(1)\\
&=\frac{1}{n}\sum\limits_{\substack {i,j=1\\i\neq j}}^n\E[D(i;t_1,n)a(i,j;s,n)D(j;t,n)a(i,j;s,n) D(i;t_3,n)]\\
&+\frac{1}{n}\sum\limits_{\substack {i=j=1}}^n\E[D(i;t_1,n)a(i,j;s,n)D(j;t,n)a(i,j;s,n) D(i;t_3,n)]+o(1).\\
\end{align*}

By Lemma \ref{infinitesimal 1}, we have $$\frac{1}{n}\sum\limits_{\substack {i=j=1}}^n\E[D(i;t_1,n)a(i,j;s,n)D(j;t,n)a(i,j;s,n) D(i;t_3,n)]=o(1).$$
Therefore, the statement is true for $m=2$. 

Suppose $m=2k$ for $k>1$.   Let ${\bf Out_2(m)}$ be the family of $\stackrel{out}{\sim}$ equivalence classes of $NC_2(m)$.  
Given a noncrossing pair partition $\pi$, we denote by $[\pi]_{\stackrel{out}{\sim}}$ the family of noncrossing pair partitions which are $\stackrel{out}{\sim}$ equivalent to $\pi$.    Then, we have
\begin{align*}
&\frac{1}{n}\sum\limits_{\substack{\pi\in NC_2(m)\\ \pi\leq \ker\bar s}}\sum\limits_{\substack{\ii\subset [n]\\ \ker\ii=OK(\pi)} }\E[Pro_n(\bar{s},\bar{t},\ii)]\\
=&\frac{1}{n}\sum\limits_{[\pi']_{\stackrel{out}{\sim}}\in {\bf Out_2(m)} }\sum\limits_{\substack{\pi\in[\pi']_{\stackrel{out}{\sim}}\\ \pi\leq \ker\bar s}}\sum\limits_{\substack{\ii\subset [n]\\ \ker\ii=OK(\pi)} }\E[Pro_n(\bar{s},\bar{t},\ii)].
\end{align*}

On the other hand, we have 
\begin{align*}
&\EE[D_{t_1}Y_{s_1}\cdots D_{t_m}Y_{s_m}D_{t_{m+1}}]\\
=&\sum\limits_{\pi\in NC_2(m)}\sum\limits_{\substack{\sigma\in NC(m+1)\\ \pi\hookrightarrow \sigma\in NC(2m+1)}}\kappa_\EE^{(\pi\hookrightarrow \sigma)}(D_{t_1},Y_{s_1},\cdots ,D_{t_m},Y_{s_m},D_{t_{m+1}})\\
=&\sum\limits_{[\pi']_{\stackrel{out}{\sim}}\in {\bf Out_2(m)} }\sum\limits_{\substack{\pi\in[\pi']_{\stackrel{out}{\sim}}}}\sum\limits_{\substack{\sigma\in NC(m+1)\\ \pi\hookrightarrow \sigma\in NC(2m+1)}}\kappa_\EE^{(\pi\hookrightarrow \sigma)}(D_{t_1},Y_{s_1},\cdots ,D_{t_m},Y_{s_m},D_{t_{m+1}}).\\
\end{align*}
Notice that that  $\kappa_\EE^{(\pi\hookrightarrow \sigma)}(D_{t_1},Y_{s_1},\cdots ,D_{t_m},Y_{s_m},D_{t_{m+1}})=0$ if $\pi\not\leq \ker \bar s$, the above equation becomes
\begin{align*}
&\EE[D_{t_1}Y_{s_1}\cdots D_{t_m}Y_{s_m}D_{t_{m+1}}]\\
=&\sum\limits_{[\pi']_{\stackrel{out}{\sim}}\in {\bf Out_2(m)} }\sum\limits_{\substack{\pi\in[\pi']_{\stackrel{out}{\sim}}\\ \pi\leq \ker\bar s}}\sum\limits_{\substack{\sigma\in NC(m+1)\\ \pi\hookrightarrow \sigma\in NC(2m+1)}}\kappa_\EE^{(\pi\hookrightarrow \sigma)}(D_{t_1},Y_{s_1},\cdots ,D_{t_m},Y_{s_m},D_{t_{m+1}}).\\
\end{align*}

Therefore,  we just need to show that the following equation holds.
\begin{align*}
&\frac{1}{n}\sum\limits_{\substack{\pi\in[\pi']_{\stackrel{out}{\sim}}\\ \pi\leq \ker\bar s}}\sum\limits_{\substack{\ii\subset [n]\\ \ker\ii=OK(\pi)} }\E[Pro_n(\bar{s},\bar{t},\ii)]\\
= &\sum\limits_{\substack{\pi\in[\pi']_{\stackrel{out}{\sim}}\\ \pi\leq \ker\bar s}}\sum\limits_{\substack{\sigma\in NC(m+1)\\ \pi\hookrightarrow \sigma\in NC(2m+1)}}\kappa_\EE^{(\pi\hookrightarrow \sigma)}(D_{t_1},Y_{s_1},\cdots ,D_{t_m},Y_{s_m},D_{t_{m+1}})+o(1).
\end{align*}

Suppose that $\{V_1=\{j_1,j_1'\},\cdots,V_l=\{j_l,j_l'\}\}$ is  the family  of outer blocks of $\pi'$ and  $V_1,\cdots, V_l$ are ordered, i.e.  for $k=1,\cdots, l-1$, $j'_k+1=j_{k+1}$.  By Lemma \ref{recursive relation of limite distribution}, we have
\begin{align*}
&\sum\limits_{\substack{\pi\in NC(m)\\ \pi\stackrel{out}{\sim} \pi'}}\sum\limits_{\substack{\sigma\in NC(m+1)\\ \pi\hookrightarrow \sigma\in NC(2m+1)}}\kappa_\EE^{(\pi\hookrightarrow \sigma)}(D_{t_1},Y_{s_1},\cdots ,D_{t_m},Y_{s_m},D_{t_{m+1}})\\
=&\EE[D_{t_1}E_{V_1}D_{t_{j'_1}} E_{V_2}\cdots E_{V_l}D_{t_{m+1}}],
\end{align*}
where  $E_{V_k}=\eta_{s_{j_k}}(E[D_{I_{j_k+1}}Y_{s_{j_k+1}}\cdots Y_{s_{j_k'-1}} D_{j_k'-1}])$ for $k=1,\cdots,l$.

On the other hand, given $\pi\in [\pi']_{\stackrel{out}{\sim}}$, let $\sigma=OK(\pi)$. 
Then, for each $k$, the restriction of $\pi$ to $[j_k+1,j_k'-1]$ is also a noncrossing pair partition.  
By Lemma \ref{inside OK}, $OK(\pi([j_k+1,j_k'-1]))=\sigma([j_k+1,j_k'])$.  
Furthermore, according to the definition of $OK$, $\{1,j_2,\cdots, j_l,m+1\}$ is a block of $\sigma$. 
If $\pi\leq \ker \bar s$, then $\pi([j_k+1,j_k'-1])\leq \ker \bar s([j_k+1,j_k'-1])$ for all $k$, where $\bar s([j_k+1,j_k'-1]$ is the subsequence $(s_{j_k+1},\cdots,s_{j'_k-1})$ of $\bar s$.
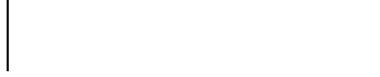
\begin{center}
\begin{tikzpicture}[scale=1]

\draw [thick] (0,0)-- (5,0);
\draw [thick] (0,0)-- (0,-1);
\draw [thick] (5,0)-- (5,-1);

\put (-7, -7) {$\cdots$};
\put (26, -14) {$\cdots$};
\put (53, -7) {$\cdots$};

\put (-1, -14) {$\bar j_k$};
\put (6, -14) {$(j_k+1)$};
\put (39, -14) {$ j_k'$};
\put (49, -14) {$\bar j_k'$};

\end{tikzpicture}
\end{center} 
$\ $

Therefore, we have a bijection from $\{\pi|\pi\in[\pi'],\ker\bar s\}$ to $ \prod\limits_{k=1}^l\{\pi_k|\pi_k\in NC_2([j_k+1,j'_k-1]),\pi_k \leq \ker \bar s([j_k+1,j_k'-1])\}$ such that
$$\pi\rightarrow\{V_1,\cdots,V_l\}\bigcup\limits_{k=1}^l \pi_k, $$
where $\pi_k$ is the restriction of $\pi$  to the interval $[j_k+1,j'_k-1]$. In addition,
$$\sigma= \{\{1,j_2,\cdots, j_l,m+1\}\}\bigcup\limits_{k=1}^l OK(\pi_k).$$ 

For $\ii\subset [n]$ such that $\ker\ii=\sigma$,  let $\ii_1=(i_2,\cdots,i_{j_1'})$ and $\ii'=(i_1, i_{j'_1+1},i_{j'_1+2}\cdots,i_{m+1})$. 
Then $$\ker\ii_1=\sigma([2,j_1'])=OK(\pi_1)$$ 
and 
$$\ker\ii'= \{\{1,j_2,\cdots, j_l,m+1\}\}\bigcup\limits_{k=2}^l OK(\pi_k).$$ 
If we fix $\pi_2,\cdots, \pi_l$
and $i_1, i_{j'_1+1},i_{j'_1+2}\cdots,i_m$, and let $\pi_1$ run over all $NC_2([j_1+1,j'_1-1]))$ such that $\pi_1\leq \ker s([j_1+1,j'_1-1]))$.  Let $V=\{ i_1, i_{j'_1+1},i_{j'_1+2}\cdots,i_m\}$. Then $|V|<m$,  we have
\begingroup
\allowdisplaybreaks
\begin{align*}
&\frac{1}{n}\sum\limits_{\substack{\pi\in[\pi']_{\stackrel{out}{\sim}}\\ \pi\leq \ker\bar s}}\sum\limits_{\substack{\ii\subset [n]\\ \ker\ii=OK(\pi)} }\E[Pro_n(\bar{s},\bar{t},\ii)]\\
=&
\sum\limits_{\substack{\pi_2\in NC_2([j_2+1,j'_1-1]) \\ \pi_2 \leq \ker \bar s([j_2+1,j_2'-1])}}
\cdots 
\sum\limits_{\substack{\pi_l\in NC_2([j_l+1,j'_l-1]) \\ \pi_l \leq \ker \bar s([j_l+1,j_l'-1])}}
\sum\limits_{\substack{\ii'\subset[n]\\ker\ii'=\{\{1,j_2,\cdots, j_l,m+1\}\}\bigcup\limits_{k=2}^l OK(\pi_k) } }\\
&
\sum\limits_{\substack{\pi_1\in NC_2([j_1+1,j'_1-1]) \\ \pi_1\leq \ker \bar s([j_1+1,j_1'-1])}}
\sum\limits_{\substack{\ii_1\subset [n]\setminus V\\ \ker\ii_1=OK(\pi_1) } } \E[Pro_n(\bar{s},\bar{t},\ii)].\\
\end{align*}
\endgroup
Fixing $\ii'$,  since $\{1,j_1'\}$ is a block of $\pi$, we have 

\begin{align*}
&\sum\limits_{\substack{\pi_1\in NC_2([j_1+1,j'_1-1]) \\ \pi_1\leq \ker \bar s([j_1+1,j_1'-1])}}
\sum\limits_{\substack{\ii_1\subset [n]\setminus V\\ \ker\ii_1=OK(\pi_1) } } \E[Pro_n(\bar{s},\bar{t},\ii)]\\
=&\sum\limits_{\substack{\pi_1\in NC_2([j_1+1,j'_1-1]) \\ \pi_1\leq \ker \bar s([j_1+1,j_1'-1])}}
\sum\limits_{\substack{\ii_1\subset [n]\setminus V\\ \ker\ii_1=OK(\pi_1)} }\E[b(i_1;t_1,n)a(i_1,i_2;s_1,n)\cdots b(i_{m};t_{m},n) a(i_{m},i_{m+1};s_m,n)b(i_{m+1};t_{m+1},n)]\\
=&\sum\limits_{\substack{\pi_1\in NC_2([j_1+1,j'_1-1]) \\ \pi_1\leq \ker \bar s([j_1+1,j_1'-1])}}
\sum\limits_{\substack{\ii_1\subset [n]\setminus V\\ \ker\ii_1=OK(\pi_1) } }\E[b(i_1;t_1,n)\left(\frac{1}{n}\eta_{s_1} (\E[b(i_2;t_2,n)\cdots b(i_{j_1';t_{j_1'},n}])\right)\\
&\cdots b(i_{m};t_{m},n) a(i_{m},i_{m+1};s_m,n)b(i_{m+1};t_{m+1},n)]\\
=&\E[b(i_1;t_1,n)\left(\sum\limits_{\substack{\pi_1\in NC_2([j_1+1,j'_1-1]) \\ \pi_1\leq \ker \bar s([j_1+1,j_1'-1])}}
\sum\limits_{\substack{\ii_1\subset [n]\setminus V\\ \ker\ii_1=OK(\pi_1) } }\eta_{s_1} (\frac{1}{n}\E[b(i_2;t_2,n)\cdots b(i_{j_1';t_{j_1'},n}])\right)\\
&\cdots b(i_{m};t_{m},n) a(i_{m},i_{m+1};s_m,n)b(i_{m+1};t_{m+1},n)].\\
\end{align*}

By Lemma \ref{partial sum}, we have 
\begin{align*}
&\sum\limits_{\substack{\pi_1\in NC_2([j_1+1,j'_1-1]) \\ \pi_1\leq \ker \bar s([j_1+1,j_1'-1])}}
\sum\limits_{\substack{\ii_1\subset [n]\setminus V\\ \ker\ii_1=OK(\pi_1) } }\frac{1}{n}\E[b(i_2;t_2,n)\cdots b(i_{j_1';t_{j_1'},n})]\\
=&\sum\limits_{\substack{\pi_1\in NC_2([j_1+1,j'_1-1]) \\ \pi_1\leq \ker \bar s([j_1+1,j_1'-1])}}
\sum\limits_{\substack{\ii_1\subset [n]\\ \ker\ii_1=OK(\pi_1) } }\frac{1}{n}\E[b(i_2;t_2,n)\cdots b(i_{j_1';t_{j_1'},n})]+o(1).
\end{align*}

By induction, \begin{align*}
\sum\limits_{\substack{\pi_1\in NC_2([j_1+1,j'_1-1]) \\ \pi_1\leq \ker \bar s([j_1+1,j_1'-1])}}
\sum\limits_{\substack{\ii_1\subset [n]\\ \ker\ii_1=OK(\pi_1) } }\frac{1}{n}\E[b(i_2;t_2,n)\cdots b(i_{j_1';t_{j_1'},n})]=\E[D_{t_2}\cdots D_{t_{j_1'}}]+o(1).
\end{align*}

Therefore, for fixed $\ii'$,
\begin{align*}
&\sum\limits_{\substack{\pi_1\in NC_2([j_1+1,j'_1-1]) \\ \pi_1\leq \ker \bar s([j_1+1,j_1'-1])}}
\sum\limits_{\substack{\ii_1\in [n]\setminus\{\ii'\}\\ \ker\ii_1=OK(\pi_1) } } \E[Pro_n(\bar{s},\bar{t},\ii)]\\
=&\E[b(i_1;t_1,n)E_{V_1}b(i_1;t_{j'_1+1})\cdots b(i_{m};t_{m},n) a(i_{m},i_{m+1};s_m,n)b(i_{m+1};t_{m+1},n)]+o(1).
\end{align*}

Repeat the previous steps to $V_2,\cdots, V_l$, we will get 

\begin{align*}
&\frac{1}{n}\sum\limits_{\substack{\pi\in[\pi']_{\stackrel{out}{\sim}}\\ \pi\leq \ker\bar s}}\sum\limits_{\substack{\ii\subset [n]\\ \ker\ii=OK(\pi)} }\E[Pro_n(\bar{s},\bar{t},\ii)]\\
=&\frac{1}{n}\sum\limits_{i_1\in[n] }\E[b(i_1;t_1,n)E_{V_1}b(i_1;t_{j'_1+1})E_{V_2}\cdots E_{V_l}b(i_1;t_{m+1},n) ]+o(1)\\
=&\E_n[f(\{\D(t,n)|t\in T\})]+o(1).
\end{align*}
where $f=X_{t_1}E_{V_1}X_{t_{j'_1+1}}E_{V_2}\cdots E_{V_l}X_{t_{m+1}}$.  It follows that

\begin{equation}\label{limit moments}
\begin{aligned}
&\frac{1}{n}\sum\limits_{\substack{\pi\in[\pi']_{\stackrel{out}{\sim}}\\ \pi\leq \ker\bar s}}\sum\limits_{\substack{\ii\subset [n]\\ \ker\ii=OK(\pi)} }\E[Pro_n(\bar{s},\bar{t},\ii)]\\
=&\EE[f(\{D_t|t\in T\})]+o(1)\\
=&\E[D_{t_1}E_{V_1}D_{t_{j'_1+1}}E_{V_2}\cdots E_{V_l}D_{t_{m+1}}]+o(1).
\end{aligned}
\end{equation}

The proof is done.
\end{proof}

\section{Asymptotic freeness in an extended matrix model}

In this section, we prove an operator-valued asymptotic freeness theorem in Dykema's extended matrix model. 
Again, for each $n\in \N$, let  $\{e(i,j;n)|i,j=1,\cdots,n\}$ be family of matrix unites of $M_n(\C)$.  
We will use the following identification for the matrix algebras $M_N(C)\otimes M_k(\C)$ and $M_{kN}(\C)$: 
$$ e(p,q;N)\otimes e(i,j;k) \rightarrow e(p+(i-1)N,q+(j-1)L;kN).$$

Again, we assume that
$$ Y(s,n)=\sum\limits_{1\leq i,j, \leq n}  a(i,j;s,n)\otimes e(i,j;n)\in \domain \otimes M_n$$
are $\range$-valued random matrices  for  $s$ taking values in some set $S$ such that satisfying Conditions Y1)-Y5) in Theorem \ref{main}.

Let $\domain_N=\domain\otimes M_N(\C)$ and let $\E_{N,1}$ be the map from $\domain\otimes M_N(\C)$ to $\range\otimes \M_N(\C)$ such that $$\E_{N,1}[(a_{p,q})_{p,q=1,\cdots,N}]= (\E[a_{p,q}])_{p,q=1,\cdots,N},$$
where $a_{p,q}\in \domain,$ for $p,q=1,\cdots,N.$

It is obvious that   $\E_{N,1}$ is a $\range\otimes \M_N(\C)$-valued conditional expectation.  
In addition, $\E_{N,1}$ is positive since $\E$ is completely positive.
Therefore, $ (\domain\otimes M_N(\C),\E_{N,1}:\domain\otimes M_N(\C)\rightarrow \range\otimes M_N(\C))$ is a $\range\otimes M_N(\C)$-valued probability space.   Fixing $k$, let  $\{A(i,j;s,k)|s\in S, 1\leq i,j \leq m\}$ be the family of  $\range\otimes M_N(\C)$-valued random variables such that 
$$A(i,j;s,k)=\sum\limits_{1\leq p,q\leq N}a(p+(i-1)N,q+(j-1)N;s,kN)\otimes e(p,q;N).$$
Since
$ Y(s,n)=\sum\limits_{1\leq i,j, \leq n}  a(i,j;s,n)\otimes e(i,j;n)\in \domain \otimes M_n$, we have that 
  $$Y(s,kN)=\sum\limits_{1\leq i,j, \leq k}  A(i,j;s,k)\otimes e(i,j;n)\in (\domain\otimes M_N(\C) )\otimes M_k(\C).$$

Then, we can define a  conditional expectation $\E_{N,k}:  (\domain\otimes M_N(\C) )\otimes M_k(\C)\rightarrow \range\otimes M_N(\C)$ as follows:
$$ \E_{N,k}[\sum\limits_{1\leq i,j \leq k}  a(i,j)\otimes e(i,j;k)]=\frac{1}{k}\sum\limits_{1\leq i \leq k}  \E_{N,1}[a(i,i)].$$
where $a(i,j)\in \domain\otimes M_N(\C). $  One can easily see that  $\E_{kN}[\cdot]=\E_{N}[\E_{N,k}[\cdot]]$.
\begin{lemma}
$A(j,i;s,k)=A(i,j;s,k)^*,$  for all $1\leq i,j\leq k$, $s\in S$.
\end{lemma}
\begin{proof}  Notice that $a(i,j;s,n)=a(i,j;s,n)^*,$  for all $1\leq i,j\leq n$, $s\in S$. We have
\begin{align*}
A(i,j;s,k)^*=&(\sum\limits_{1\leq p,q\leq N}a(p+(i-1)N,q+(j-1)N;s,kN)\otimes e(p,q;N))^*\\
=&\sum\limits_{1\leq p,q\leq N}[a(p+(i-1)N,q+(j-1)N;s,kN)]^*\otimes [e(p,q;N)]^*\\
=&\sum\limits_{1\leq p,q\leq N}a(q+(j-1)N,p+(i-1)N;s,kN)\otimes e(q,p;N)\\
=&\A(j,i;s,k).
\end{align*}
The proof is done.
\end{proof}

\begin{lemma}
$\E_{N,1}[A(i,j;s,k)]=0,$ for all $1\leq i,j\leq k$, $s\in S$,
\end{lemma}
\begin{proof}
It follows the fact that $\E[a(i,j;s,n))]=0,$ for all $1\leq i,j\leq n$, $s\in S$.
\end{proof}

\begin{lemma}\label{Matrix cpm}\normalfont Let $I_N$ be the unit of $\range\otimes M_N(\C)$. Then, for $1\leq i,j\leq k$, $s\in S$,
$$\E_{N,1}[A(i,j;s,k) (b_{p,q})_{p,q=1,\cdots,N} A(j,i;s,k)]=\frac{1}{k}\eta_s(\sum\limits_{i=1}^N b_{p,p})\otimes I_N,$$ 
 for all $(b_{i,j})_{i,j=1,\cdots,N} \in\domain\otimes M_N(\C)$. Furthermore, $\E_{N,1}[A(i,j;s,k) \bullet A(i,j;s,k)]$ is a completely positive map from $\range\otimes M_N(\C)$ to $\range\otimes M_N(\C)$.
\end{lemma}
\begin{proof}
For $1\leq s,t\leq N$, the $(s,t)$-th entry of $\E_{N,1}[A(i,j;s,k) (b_{p,q})_{p,q=1,\cdots,N} A(i,j;s,k)]$ is 
$$\E[\sum\limits_{1\leq p,q\leq N}a(s+(i-1)N,p+(j-1)N;s,kN)b_{p,q}a(q+(j-1)N,t+(i-1)N;s,kN)].$$
Since the family $\{a(i,j;n,s)|s\in S, 1\leq i\leq j \leq m\}$ of random variables are $\range-$valued conditionally independent, $\E[a(i,j;s,n))]=0,$ for all $1\leq i,j\leq n$, $s\in S$ and  $\E[a(i,j;s,n)\bullet a(j,i;s,n))]=\frac{1}{n}\eta_s(\bullet)$, for all $1\leq i,j\leq n$, $s\in S$, we have

\begin{align*}
&E[a(s+(i-1)N,p+(j-1)N;s,kN)b_{p,q}a(q+(j-1)N,t+(i-1)N;s,kN)]\\
=&\left\{
\begin{array}{lc}
\frac{1}{k}\eta_s(b_{p,p})& \text{if $s=t$ and $p=q$} \\
0&\text{otherwise.}
\end{array}\right.
\end{align*}
\end{proof}

Therefore,  the $(s,t)$-th entry of $\E_{N,1}[A(i,j;s,k) (b_{p,q}))_{p,q=1,\cdots,N} A(i,j;s,k)]$ is
 $$ \delta_{s,t}\frac{1}{k}\eta_s(\sum\limits_{i=1}^N b_{p,p}),$$
which is what we want to show.

\begin{lemma}
For each $m$, there exists an $\MM_m>0$ such that  $$\sup\limits_{\substack{s_1,...,s_m\in S\\ 1\leq i_1,\cdots i_m, j_1,\cdots, j_m
\leq n}}\|{A(i_1,j_1;s_1,k)b_{N,1}A(i_2,j_2;s_2,k) b_{N,2}\cdots b_{N,m-1}A(i_m,j_m;s_m,k)}\|\leq \MM_n^{-m/2}\prod\limits_{k=1}^{m-1}\|b_k\|,$$
where $ b_{N,1},\cdots, b_{N,m-1}\in \range\otimes M_N{\C}$. 
\end{lemma}
\begin{proof} 
Assume that ${A(i_1,j_1;s_1,k)b_{N,1}A(i_2,j_2;s_2,k) b_{N,2}\cdots b_{N,m-1}A(i_m,j_m;s_m,k)}=(a_{p,q})_{p,q=1,\cdots,N}$, then 
$$\| ((a_{p,q})_{p,q=1,\cdots,N}\|\leq N\max\limits_{p,q=1,\cdots,N}{\{\|a_{p,q}\|\}}.$$
For each $l$,  assume that  $b_{N,l}=(b(p,q;N,l))_{p,q=1,\cdots,N}$, then 
$$ \|b(p,q;N,l)\|\leq \| b_{N,l}\|.$$
By direct computations, we have that $a_{p,q}$ a sum of $N^{2m-2}$ monomials in the form of 
$$a(\cdots)b(\dots;N,1)a(\cdots)\cdots b(\dots;N,m-1)a(\cdots). $$
On the other hand, we have that  there exists an $M_m>0$ such that  $$\sup\limits_{\substack{s_1,...,s_m\in S\\ 1\leq i_1,\cdots i_m, j_1,\cdots, j_m
\leq n}}\|{a(i_1,j_1;s_1,n)b_1a(i_2,j_2;s_2,n)b_2\cdots b_{m-1}a(i_m,j_m;s_m,n)}\|\leq M_mn^{-m/2}\prod\limits_{l=1}^{m-1}\|b_l\|,$$
for $b_1,\cdots,b_{m-1}\in \range$. Therefore, we have 
\begin{align*}
&\|a(\cdots)b(\dots;N,1)a(\cdots)\cdots b(\dots;N,m-1)a(\cdots)\|\\
\leq& M_mn^{-m/2}\prod\limits_{l=1}^{m-1}\|b(\dots;N,l)\|\\
\leq& M_mn^{-m/2}\prod\limits_{l=1}^{m-1}\|b_{N,l}\|. 
\end{align*}
Therefore, $$\|a_{p,q}\|\leq N^{2m-2}M_mn^{-m/2}\prod\limits_{l=1}^{m-1}\|b_{N,l}\|.$$
Let $\MM_m=N^{2m-1}M_m$. Then statement follows.
\end{proof}

\begin{lemma}
The family $\{A(i,j;s,k)|s\in S, 1\leq i\leq j \leq m\}$ of random variables are $\range\otimes M_N(C)$-valued conditionally independent.
\end{lemma}
\begin{proof}
Let $S_{N1},S_{N2} \subset \{A(i,j;s,k)|s\in S, 1\leq i\leq j \leq m\}$ such that $S_{N1}\cap S_{N2}=\emptyset$. For $l=1,2$, define 
$$S_l=\{a(p,q;s, kN)| a(p,q;s, kN)\,\text{is an entry of}\,A(i,j;s,k), A(i,j;s,k)\in S_l  \}.$$
By the definition of $A(i,j;s,k)$, we have that $S_1\cap S_2=\emptyset$. Given $A_1,A_3\in Alg_\range\{S_{N1}\}$ and $A_2 \in Alg_\range\{S_{N1}\}$.

For $t=1,2,3$, assume that $A_t=(A(p,q;t))_{p,q=1,\cdots,N}$, then $A(p,q;1),A(p,q;3)\in  Alg_\range\{S_{1}\}$ and $A(p,q;2)\in  Alg_\range\{S_{2}\}$
for all $p,q=1,\cdots N.$  Therefore,
   $$\E[A(p_1,p_2;1)A(p_2,p_3;t2)A(p_3,p_4;2)]=\E[A(p_1,p_2;1)\E[A(p_2,p_3;t2)]A(p_3,p_4;2)].$$
It follows that 
    $$\E_{N,1}[A_1A_2A_3]=\E_{N,1}[A_1\E_{N,1}[A_2]A_3].$$
    Since $A_1,A_2,A_3$ are arbitrary,  the proof is done.
\end{proof}

In summary,  as a family of $\range\otimes M_N(\C)$-valued random variables,  the family of random variables $\{A(i,j;s,k)|s\in S, 1\leq i, j \leq m\}$ satisfy the conditions Y1) to Y5) in Theorem \ref{main}.\\

Let $D(t,kN)=\sum\limits_{1\leq i \leq k}  D(i;t,k)\otimes e(i,i,k)$, where  $D(i;t,k)\in\range\otimes M_N(\C)$ and $t$ takes values in some index set $T$.  Suppose that $(\{D(t,kN)|t\in T\})$ satisfying the following conditions:
\begin{itemize}
\item[D1')] the joint distribution of $\{D(t,kN)\}_{t\in T}$ converges weakly in norm  as $\range$-valued random variables, namely with respect to $\E_{kN}$.
\item[D2')] $$\lim\limits_{n\rightarrow \infty}\frac{\|D(t,n)\|^l}{k}=0,$$ for all $t\in T$, $l\in\N$.
\item[D3')] For all $t\in T$ and $l\in\N$, we have$$ \limsup\frac{\sum\limits_{i=1}^n\|D(t,kN)\|_{\range\otimes M_N(\C)}^l}{k}< \infty,$$ 
where $\|\cdot\|_{\range\otimes M_N(\C)}$ is the $C^*$-norm on  $\range\otimes M_N(\C)$.
\end{itemize}

\begin{remark} \normalfont 
One should be careful that Condition D1') does not imply that   $\{D(t,kN)\}_{t\in T}$ converges weakly in norm as $\range\otimes M_n(\C)$-valued random variables with respect to $\E_{N,k}$.  For example, let $N=2$ and $D(i;k)=(-1)^k[1_\range\otimes e(1,2;2)+1_\range\otimes e(2,1;2)]$. Then the sequence $(D(kN)=\sum\limits_{1\leq i \leq k}  D(i;k)\otimes e(i,i,k))_{k\in \N}$ converges weakly in norm with respect to $\E_{2k}$ but not for $\E_{2,k}$.
\end{remark}

Therefore,  as $\range\otimes M_N(\C)$-valued random variables,  $(\{ Y(s,kN)| s\in S\}, \{D(t,kN)\}_{t\in T})$ satisfy the Conditions Y1) to Y5), D2) D3) in Theorem \ref{main}. Following the proof of Theorem \ref{main} and the fact that 
$\{D(t,kN)\}_{t\in T}$ converges weakly in norm  with respect to $\E_{kN}$. We have the following theorem.

\begin{theorem}\label{matrix model}\normalfont
Let 
$$ Y(s,n)=\sum\limits_{1\leq i,j, \leq n}  a(i,j;s,n)\otimes e(i,j;n)\in \domain \otimes M_n$$
be $\range$-valued random matrices  for  $s$ taking values in some set $S$ such that satisfying Y1)-Y5).

Fixing  $N$, for  $t$ taking values in some set $T$ and $k\in \N$, let $D(t,kN)=\sum\limits_{1\leq i \leq n}  D(i;t,k)\otimes e(i,i,k)$, where  $D(i;t,k)\in\range\otimes M_N(\C)$.  Suppose that $(\{D(t,kN)|t\in T\})$ satisfying the following conditions
\begin{itemize}
\item[D1')] the joint distribution of $\{D(t,kN)\}_{t\in T}$ converges weakly in norm  with respect to $\E_{kN}$.
\item[D2')] $$\lim\limits_{n\rightarrow \infty}\frac{\|D(t,n)\|^l}{k}=0,$$ for all $t\in T$, $l\in\N$.
\item[D3')] For all $t\in T$ and $l\in\N$, we have$$ \limsup\frac{\sum\limits_{i=1}^n\|D(t,kN)\|_{\range\otimes M_N(\C)}^l}{k}< \infty,$$ 
where $\|\cdot\|_{\range\otimes M_N(\C)}$ is the $C^*$-norm on  $\range\otimes M_N(\C)$.
\end{itemize}

Then the joint distributions of the  family of sets of random variables 
$$\{Y(s,kN)\}_{s\in S}\cup\{D(t,kN)|t\in T\}$$
with respect to $\E_{kN}$ converge weakly in norm to the joint distribution of  the family of  $$\{Y_s\}_{s\in S}\cup\{D_t|t\in T\}$$  such that the family of subsets $\{ \{Y_s\}|s\in S \}\cup\{\{D_t|t\in T\}\}$ are freely independent. 
Moreover,  for each $s\in S$,  the distribution of $Y_s$ is a $\range$-valued semicircular law with variance $\eta_s.$
\end{theorem}

\section{Operator valued Boolean random matrices} 

In this section, we  consider limit laws of  random matrices having Boolean independent entries.  
Again, we  use the notation  $Pro_n(\bar{s},\bar{t},\ii)$ short for
$$b(i_1;t_1,n)a(i_1,i_2;s_1,n) b(i_2; t_2,n)a(i_2,i_3;s_2,n)\cdots b(i_{m};t_{m},n) a(i_{m},i_{m+1};s_m,n)b(i_{m+1};t_{m+1},n), $$
where $\bar{s}=(s_1,\cdots, s_m)\subset S$, $ \bar{t}=(t_1,\cdots, t_{m+1})\subset T$ and $\ii=(i_1,\cdots,i_{m+1})\subset [n]$.

\begin{lemma} \label{nonvanishing}\normalfont
Let $\{a(i,j;n,s)|s\in S, 1\leq i\leq j \leq m\}$ be a family of $\range$-valued  Boolean independent random variables such that  $\E[a(i,j;s,n))]=0,$ for all $1\leq i,j\leq n$, $s\in S$. Suppose that $\ker\ii=OK(\pi)$ such that $\pi\in NC_2(m)$.  Then $\E[Pro_n(\bar{s},\bar{t},\ii)]=0$
unless $\pi$ is an interval pair partition.
\end{lemma}
\begin{proof}
If $\pi$ is not an interval partition, then there is a $p$ such that $p\not\sim_\pi p+1,p-1$. Therefore  $a(i_{p},i_{p+1};s_p,n)$ is boolean independent from $a(i_{p+1},i_{p+2};s_{p+1},n)$ and $a(i_{p-1},i_{p};s_{p-1},n)$. Since $\E[a(i_{p},i_{p+1};s_p,n)]=0$, the equation follows.
 \end{proof}

\begin{remark}\normalfont
It is obvious that $IN_2(m)=\emptyset$ when $m$ is odd and $IN_2(m)$ contains exactly one element $1_{IN_2(m)}=\{\{1,2\},\cdots,\{m-1,m\}\}$  when $m$ is even. 
\end{remark}

\begin{theorem}\label{main Boolean}\normalfont

 Let 
$$ Z(s,n)=\sum\limits_{1\leq i,j, \leq n}  a(i,j;s,n)\otimes e(i,j;n)\in \domain \otimes M_n$$

be $\range$-valued random matrices  for  $s$ taking values in some set $S$ such that 
\begin{itemize}
\item[Y1)]  $a(i,j;s,n)=a(i,j;s,n)^*,$  for all $1\leq i,j\leq n$, $s\in S$,
\item[Y2)]  $\E[a(i,j;s,n))]=0,$ for all $1\leq i,j\leq n$, $s\in S$,
\item[Y3)]  $\E[a(i,j;s,n)\bullet a(j,i;s,n))]=\frac{1}{n}\eta_s(\bullet)$ is a completely positive map from $\range$ to $\range$, for all $1\leq i,j\leq n$, $s\in S$,
\item[Y4)] for each $m$, there exists an $M_m>0$ such that  $$\sup\limits_{\substack{s_1,...,s_m\in S\\ 1\leq i_1,\cdots i_m, j_1,\cdots, j_m
\leq n}}\|{a(i_1,j_1;s_1,n)b_1a(i_2,j_2;s_2,n)b_2\cdots b_{m-1}a(i_m,j_m;s_m,n)}\|\leq M_mn^{-m/2}\prod\limits_{k=1}^{m-1}\|b_k\|,$$
\item[Y5)] the family $\{a(i,j;n,s)|s\in S, 1\leq i\leq j \leq m\}$ of random variables are $\range-$valued {\bf Boolean independent}.
\end{itemize}

Then the joint distributions of $\{Z(s,n)\}_{s\in S}$ converge to the distribution of the family of Boolean independent Bernoulli random variables $\{Z_s\}_{s\in S}$.

In addition,  let 
$D(t,n)=\sum\limits_{1\leq i \leq n}  b(i;t,n)\otimes e(i,i,n)\in\range\otimes M_n$ for  $t$ taking values in some set $T$ such that
\begin{itemize}
\item[D1)] the joint distribution of $(D(t,n))_{t\in T}$ converges weakly in norm to the joint distribution of $(D_t)_{t\in T}$.
\item[D2)] $$\lim\limits_{n\rightarrow \infty}\frac{\|D(t,n)\|^k}{n}=0,$$ for all $t\in T$, $k\in\N$.
\item[D3)] $$ \limsup\frac{\sum\limits_{i=1}^n\|b(i;t,n)\|^k}{n}< \infty,$$ for all $t\in T$, $k\in\N$.
\end{itemize}
  
Suppose that $\{Z_s\}_{s\in S}$ and $(D_t)_{t\in T}$ are from the $\range$-valued probability space $(\Domain,\EE:\Domain\rightarrow \range)$. Then, we have that 
\begin{equation}\label{Boolean matrices}
\begin{aligned}
&\EE[P_1Z_{s_1}P_2Z_{s_2}\cdots P_m Z_{s_m} P_{m+1}]\\
=&\left\{\begin{array}{lc}
\EE[P_1 \eta_{s_1}(\EE[P_2])P_3\eta_{s_3}(\EE[P_4])\cdots P_{m-1}\eta_{s_m-1}(\EE[P_m])P_{m+1} ]& \text{if $\ker \bar{s}\geq 1_{IN_2(m)}$}\\
 0& \text{otherwise,}
\end{array}\right.
\end{aligned}
\end{equation}
where $P_1,\cdots,P_{m+1}$ are from the algebra generated by $(D_t)_{t\in T}$ and $\range$, $\bar{s}$ is the sequence $(s_1,\cdots,s_m).$
\end{theorem}
\begin{proof} 
Notice that $\{Z_s\}_{s\in S}$ is a family of a Boolean independent Bernoulli random variables if Equation \ref{Boolean matrices} holds when $P_1\cdots,P_{m+1}$ are chose to be constants. Therefore, it suffices to prove Equation \ref{Boolean matrices}.  According the work we did in the beginning of Section 4, it suffices to show that
\begin{equation}\label{Boolean matrices}
\begin{aligned}
&\lim\limits_{n\rightarrow \infty}\frac{1}{n}\sum\limits_{\ii\subset [n]}\E[Pro_n(\bar{s},\bar{t},\ii)]\\
=&\left\{\begin{array}{lc}
\EE[D_{t_1} \eta_{s_1}(\EE[D_{t_2}])\cdots D_{t_{m-1}}\eta_{s_{m-1}}(\EE[D_{t_m}])D_{t_{m+1}} ]& \text{if $\ker \bar{s}\geq 1_{IN_2(m)}$}\\
 0& \text{otherwise,}
\end{array}\right.
\end{aligned}
\end{equation}

By the simplifications we used in the proof of Theorem \ref{main}, we have 
$$
\frac{1}{n}\sum\limits_{\ii\subset [n]}\E[Pro_n(\bar{s},\bar{t},\ii)]=\frac{1}{n}\sum\limits_{\substack{\pi=1_{IN_2(m)}\\ \pi\leq \ker\bar s}}\sum\limits_{\substack{\ii\subset [n]\\ \ker\ii=OK(\pi)} }\E[Pro_n(\bar{s},\bar{t},\ii)]+o(1).$$

Therefore, $$\lim\limits_{n\rightarrow \infty}\frac{1}{n}\sum\limits_{\ii\subset [n]}\E[Pro_n(\bar{s},\bar{t},\ii)]=0$$ when 
$\ker\bar{s}\not\leq  1_{IN_2(m)}$.
When $\ker\bar{s}\leq  1_{IN_2(m)}$, notice that $[1_{IN_2(m)}]_{\stackrel{out}{\sim}}=\{1_{IN_2(m)}\}$, the statements follows Equation \ref{limit moments}.
\end{proof}

\bibliographystyle{plain}

\bibliography{references}

\begin{thebibliography}{10}

\bibitem{AGZ}
Greg~W. Anderson, Alice Guionnet, and Ofer Zeitouni.
\newblock {\em An introduction to random matrices}, volume 118 of {\em
  Cambridge Studies in Advanced Mathematics}.
\newblock Cambridge University Press, Cambridge, 2010.

\bibitem{BPV}
S.~T. Belinschi, M.~Popa, and V.~Vinnikov.
\newblock On the operator-valued analogues of the semicircle, arcsine and
  {B}ernoulli laws.
\newblock {\em J. Operator Theory}, 70(1):239--258, 2013.

\bibitem{BMS}
Serban~T. Belinschi, Tobias Mai, and Roland Speicher.
\newblock Analytic subordination theory of operator-valued free additive
  convolution and the solution of a general random matrix problem.
\newblock {\em J. Reine Angew. Math.}, 732:21--53, 2017.

\bibitem{BG}
G.~Ben~Arous and A.~Guionnet.
\newblock Large deviations for {W}igner's law and {V}oiculescu's
  non-commutative entropy.
\newblock {\em Probab. Theory Related Fields}, 108(4):517--542, 1997.

\bibitem{BKS}
Marek Bo\.zejko, Burkhard K\"ummerer, and Roland Speicher.
\newblock {$q$}-{G}aussian processes: non-commutative and classical aspects.
\newblock {\em Comm. Math. Phys.}, 185(1):129--154, 1997.

\bibitem{BS3}
Marek Bo\.zejko and Roland Speicher.
\newblock An example of a generalized {B}rownian motion. {II}.
\newblock In {\em Quantum probability \&\ related topics}, QP-PQ, VII, pages
  67--77. World Sci. Publ., River Edge, NJ, 1992.

\bibitem{BS2}
Marek Bo\.zejko and Roland Speicher.
\newblock Completely positive maps on {C}oxeter groups, deformed commutation
  relations, and operator spaces.
\newblock {\em Math. Ann.}, 300(1):97--120, 1994.

\bibitem{BS1}
Marek Bo\.zejko and Roland Speicher.
\newblock Interpolations between bosonic and fermionic relations given by
  generalized {B}rownian motions.
\newblock {\em Math. Z.}, 222(1):135--159, 1996.

\bibitem{Bu}
P.~Busch.
\newblock Quantum states and generalized observables: a simple proof of
  {G}leason's theorem.
\newblock {\em Phys. Rev. Lett.}, 91(12):120403, 4, 2003.

\bibitem{CS}
Stephen Curran and Roland Speicher.
\newblock Quantum invariant families of matrices in free probability.
\newblock {\em J. Funct. Anal.}, 261(4):897--933, 2011.

\bibitem{Dy3}
Ken Dykema.
\newblock Free products of hyperfinite von {N}eumann algebras and free
  dimension.
\newblock {\em Duke Math. J.}, 69(1):97--119, 1993.

\bibitem{Dy2}
Ken Dykema.
\newblock On certain free product factors via an extended matrix model.
\newblock {\em J. Funct. Anal.}, 112(1):31--60, 1993.

\bibitem{Dy1}
Ken Dykema.
\newblock Interpolated free group factors.
\newblock {\em Pacific J. Math.}, 163(1):123--135, 1994.

\bibitem{Gl}
Andrew~M. Gleason.
\newblock Measures on the closed subspaces of a {H}ilbert space.
\newblock {\em J. Math. Mech.}, 6:885--893, 1957.

\bibitem{HHH}
Michal Horodecki, Pawel Horodecki, and Ryszard Horodecki.
\newblock Separability of mixed states: necessary and sufficient conditions.
\newblock {\em Phys. Lett. A}, 223(1-2):1--8, 1996.

\bibitem{Ko}
Claus K{\"o}stler.
\newblock A noncommutative extended de {F}inetti theorem.
\newblock {\em J. Funct. Anal.}, 258(4):1073--1120, 2010.

\bibitem{Le}
Franz Lehner.
\newblock Cumulants in noncommutative probability theory. {I}. {N}oncommutative
  exchangeability systems.
\newblock {\em Math. Z.}, 248(1):67--100, 2004.

\bibitem{Liu}
Weihua Liu.
\newblock A noncommutative de {F}inetti theorem for boolean independence.
\newblock {\em J. Funct. Anal.}, 269(7):1950--1994, 2015.

\bibitem{MS}
James~A. Mingo and Roland Speicher.
\newblock {\em Free probability and random matrices}, volume~35 of {\em Fields
  Institute Monographs}.
\newblock Springer, New York; Fields Institute for Research in Mathematical
  Sciences, Toronto, ON, 2017.

\bibitem{NS}
Alexandru Nica and Roland Speicher.
\newblock {\em Lectures on the combinatorics of free probability}, volume 335
  of {\em London Mathematical Society Lecture Note Series}.
\newblock Cambridge University Press, Cambridge, 2006.

\bibitem{NC}
Michael~A. Nielsen and Isaac~L. Chuang.
\newblock {\em Quantum computation and quantum information}.
\newblock Cambridge University Press, Cambridge, 2000.

\bibitem{Pa}
Vern Paulsen.
\newblock {\em Completely bounded maps and operator algebras}, volume~78 of
  {\em Cambridge Studies in Advanced Mathematics}.
\newblock Cambridge University Press, Cambridge, 2002.

\bibitem{PV}
Mihai Popa and Victor Vinnikov.
\newblock Non-commutative functions and the non-commutative free
  {L}\'evy-{H}in\v cin formula.
\newblock {\em Adv. Math.}, 236:131--157, 2013.

\bibitem{PS}
Sorin Popa and Dimitri Shlyakhtenko.
\newblock Representing interpolated free group factors as group factors.
\newblock {\em arXiv:1805.10707}.

\bibitem{Ru}
Florin Radulescu.
\newblock The fundamental group of the von {N}eumann algebra of a free group
  with infinitely many generators is {${\bf R}_+\setminus 0$}.
\newblock {\em J. Amer. Math. Soc.}, 5(3):517--532, 1992.

\bibitem{RE}
N.~Raj Rao and Alan Edelman.
\newblock The polynomial method for random matrices.
\newblock {\em Found. Comput. Math.}, 8(6):649--702, 2008.

\bibitem{Ry}
\O~yvind Ryan.
\newblock On the limit distributions of random matrices with independent or
  free entries.
\newblock {\em Comm. Math. Phys.}, 193(3):595--626, 1998.

\bibitem{Sh8}
Dimitri Shlyakhtenko.
\newblock Random {G}aussian band matrices and freeness with amalgamation.
\newblock {\em Internat. Math. Res. Notices}, (20):1013--1025, 1996.

\bibitem{Sh7}
Dimitri Shlyakhtenko.
\newblock Limit distributions of matrices with bosonic and fermionic entries.
\newblock In {\em Free probability theory ({W}aterloo, {ON}, 1995)}, volume~12
  of {\em Fields Inst. Commun.}, pages 241--252. Amer. Math. Soc., Providence,
  RI, 1997.

\bibitem{Sh5}
Dimitri Shlyakhtenko.
\newblock Some applications of freeness with amalgamation.
\newblock {\em J. Reine Angew. Math.}, 500:191--212, 1998.

\bibitem{Sh4}
Dimitri Shlyakhtenko.
\newblock {$A$}-valued semicircular systems.
\newblock {\em J. Funct. Anal.}, 166(1):1--47, 1999.

\bibitem{Sp}
Roland Speicher.
\newblock On universal products.
\newblock In {\em Free probability theory ({W}aterloo, {ON}, 1995)}, volume~12
  of {\em Fields Inst. Commun.}, pages 257--266. Amer. Math. Soc., Providence,
  RI, 1997.

\bibitem{Sp1}
Roland Speicher.
\newblock Combinatorial theory of the free product with amalgamation and
  operator-valued free probability theory.
\newblock {\em Mem. Amer. Math. Soc.}, 132(627):x+88, 1998.

\bibitem{SW}
Roland Speicher and Reza Woroudi.
\newblock Boolean convolution.
\newblock In {\em Free probability theory ({W}aterloo, {ON}, 1995)}, volume~12
  of {\em Fields Inst. Commun.}, pages 267--279. Amer. Math. Soc., Providence,
  RI, 1997.

\bibitem{VDN}
D.~V. Voiculescu, K.~J. Dykema, and A.~Nica.
\newblock {\em Free random variables}, volume~1 of {\em CRM Monograph Series}.
\newblock American Mathematical Society, Providence, RI, 1992.
\newblock A noncommutative probability approach to free products with
  applications to random matrices, operator algebras and harmonic analysis on
  free groups.

\bibitem{V7}
Dan Voiculescu.
\newblock Symmetries of some reduced free product {$C^\ast$}-algebras.
\newblock In {\em Operator algebras and their connections with topology and
  ergodic theory ({B}u\c steni, 1983)}, volume 1132 of {\em Lecture Notes in
  Math.}, pages 556--588. Springer, Berlin, 1985.

\bibitem{V6}
Dan Voiculescu.
\newblock Limit laws for random matrices and free products.
\newblock {\em Invent. Math.}, 104(1):201--220, 1991.

\bibitem{V5}
Dan Voiculescu.
\newblock The analogues of entropy and of {F}isher's information measure in
  free probability theory. {I}.
\newblock {\em Comm. Math. Phys.}, 155(1):71--92, 1993.

\bibitem{V4}
Dan Voiculescu.
\newblock The analogues of entropy and of {F}isher's information measure in
  free probability theory. {II}.
\newblock {\em Invent. Math.}, 118(3):411--440, 1994.

\bibitem{V3}
Dan Voiculescu.
\newblock Operations on certain non-commutative operator-valued random
  variables.
\newblock {\em Ast\'erisque}, (232):243--275, 1995.
\newblock Recent advances in operator algebras (Orl\'eans, 1992).

\bibitem{V9}
Dan Voiculescu.
\newblock Free analysis questions. {I}. {D}uality transform for the coalgebra
  of {$\partial_{X\colon B}$}.
\newblock {\em Int. Math. Res. Not.}, (16):793--822, 2004.

\bibitem{V8}
Dan-Virgil Voiculescu.
\newblock Free analysis questions {II}: the {G}rassmannian completion and the
  series expansions at the origin.
\newblock {\em J. Reine Angew. Math.}, 645:155--236, 2010.

\bibitem{von}
John von Neumann.
\newblock {\em Mathematical foundations of quantum mechanics}.
\newblock Princeton Landmarks in Mathematics. Princeton University Press,
  Princeton, NJ, 1996.
\newblock Translated from the German and with a preface by Robert T. Beyer,
  Twelfth printing, Princeton Paperbacks.

\bibitem{Wig}
Eugene~P. Wigner.
\newblock Characteristic vectors of bordered matrices with infinite dimensions.
\newblock {\em Ann. of Math. (2)}, 62:548--564, 1955.

\bibitem{Wis}
John Wishart.
\newblock The generalised product moment distribution in samples from a normal
  multivariate population.
\newblock volume 20A, pages 32--52. 1928.

\end{thebibliography}

\noindent Department of Mathematics\\
Indiana University	\\
Bloomington, IN 47401, USA\\
E-MAIL: liuweih@indiana.edu \\

\end{document}